\documentclass[aos,preprint,reqno]{imsart}
\setattribute{journal}{name}{}
\RequirePackage{tikz}

\usepackage[round,colon,authoryear]{natbib}
\usepackage{subcaption}
\RequirePackage[OT1]{fontenc}
\RequirePackage{amsthm,amsmath,bm,natbib,graphicx,enumitem}
\RequirePackage[colorlinks,citecolor=blue,urlcolor=blue]{hyperref}
\RequirePackage{hypernat}
\usepackage{amssymb,tabularx,multicol,multirow,booktabs,csquotes}
\usepackage[top=1in, bottom=1in, left=1in, right=1in]{geometry}
\usepackage{comment}

\newtheorem{theorem}{Theorem}
\newtheorem{corollary}{Corollary}
\newtheorem{lemma}{Lemma}

\newtheorem{remark}{Remark}

\newtheorem{lem}[lemma]{Lemma}

\DeclareMathOperator{\sech}{sech}

\def \bbeta{\boldsymbol{\beta}}
\def \bx{\mathbf{x}}
\def \by{\mathbf{y}}
\def \bQ{\mathbf{Q}}
\def \bX{\mathbf{X}}

\def \be{\begin{align*}}
\def \ee{\end{align*}}

\def \I{\mathcal{I}}

\def \E{\mathbb{E}}
\def \P{\mathbb{P}}

\def \CG {\mathrm{Cov}_{\beta,\mathbf{Q},\boldsymbol{\mu}}}
\def \VG {\mathrm{Var}_{\beta,\mathbf{Q},\boldsymbol{\mu}}}

\def \bmu{\pmb{\mu}}
\def \R{\mathbb{R}}

\def \bd {\overline{d}}

\def \PJ {\P_{\mathbf{Q},\bmu}}

\def \bJ {\mathbf{J}}
\def \btQ {\tilde{\mathbf{Q}}}
\def \bI {\mathbf{I}_n}

\def \bDL {\boldsymbol{\Delta}^{(n)}}
\def \CZ {\mathrm{Cov}_{\beta,\bQ,\mathbf{0}}}
\def \VZ {\mathrm{Var}_{\beta,\bQ,\mathbf{0}}}
\def \PG {\P_{\beta,\mathbf{Q},\boldsymbol{\mu}}}
\def \EG {\E_{\beta,\mathbf{Q},\boldsymbol{\mu}}}
\def \EZ {\E_{\beta,\mathbf{Q},\boldsymbol{0}}}

\def \bmme {\mathbb{E}_{\boldsymbol{\mu}}}

\def \mm {\mathbf{m}}
\def \mc {\mathbf{c}}
\def \PZ {\mathbb{P}_{\beta,\mathbf{Q},{\bf 0}}}
\def \ZZ {Z_n(\beta,\mathbf{Q},{\bf 0})}

\begin{document}
\begin{frontmatter}
	\title{Detecting Structured Signals in Ising Models}
	\runtitle{Detection Thresholds for Ising Models}
		\thankstext{m2}{The research of Sumit Mukherjee was supported in part by NSF Grant DMS-1712037.}
	\thankstext{m4}{The research of Ming Yuan was supported in part by NSF Grant DMS-2015285.}

	\begin{aug}
			\author{\fnms{Nabarun} \snm{Deb}\thanksref{m3}\ead[label=e4]{ nd2560@columbia.edu}},
		\author{\fnms{Rajarshi} \snm{Mukherjee}\thanksref{m1}\ead[label=e2]{ram521@mail.harvard.edu}},
		\author{\fnms{Sumit} \snm{Mukherjee}\thanksref{m2}\ead[label=e1]{sm3949@columbia.edu}},
		\and
		\author{\fnms{Ming} \snm{Yuan}\thanksref{m4}\ead[label=e3]{ming.yuan@columbia.edu}}

		\affiliation{Harvard University\thanksmark{m1} and Columbia University \thanksmark{m2}\thanksmark{m4}\thanksmark{m3} }
		
		\address{Department of Biostatistics,\\
		655 Huntington Avenue, Boston, MA- 02115. \\
		E-mail: \href{mailto:ram521@mail.harvard.edu }{ram521@mail.harvard.edu } }

		\address{Department of Statistics\\
			1255 Amsterdam Avenue\\
			New York, NY-10027. \\
			E-mail: \href{mailto:nd2560@columbia.edu }{nd2560@columbia.edu } \\
			E-mail: \href{mailto:sm3949@columbia.edu }{sm3949@columbia.edu } \\
			E-mail: \href{mailto:ming.yuan@columbia.edu }{ming.yuan@columbia.edu } }
		

	\end{aug}

\begin{abstract} 
	In this paper we study the effect of dependence on detecting a class of signals in Ising models, where the signals are present in a structured way. Examples include Ising Models on lattices, and Mean-Field type Ising Models (Erd\H{o}s-R\'{e}nyi, Random regular, and dense graphs). Our results rely on correlation decay and mixing type behavior for Ising Models, and demonstrate the beneficial behavior of criticality in detection of strictly lower signals. As a by-product of our proof technique, we develop sharp control on mixing and spin-spin correlation for several Mean-Field type Ising Models in all regimes of temperature -- which might be of independent interest.
	
\end{abstract}


\end{frontmatter}

\section{Introduction} 
Let $\bX=(X_1,\ldots,X_n)^\top\in \{\pm 1\}^n$ be a random vector with the joint distribution of $\bX$ given by an Ising model defined as:
\begin{align}
\P_{\beta, \bQ,\bmu}(\bX=\bx):=\frac{1}{{\color{black} Z_n(\beta,\mathbf{Q}, \mathbf{\bmu})}}\exp{\left(\frac{\beta}{2}\bx^\top\mathbf{Q} \bx+\bmu^\top\bx\right)},\qquad \forall \bx \in \{\pm 1\}^n.
\label{eqn:general_ising}
\end{align}
Here $\mathbf{Q}$ is an $n \times n$ symmetric matrix with $0$'s on the diagonal, $\bmu:=(\mu_1,\ldots,\mu_n)^\top\in \mathbb{R}^{n}$ is an unknown parameter vector to be referred to as the external magnetization vector, $\beta\in \mathbb{R}$ is a real number usually referred to as the ``inverse temperature", and {\color{black} $Z_n(\beta,\mathbf{Q}, \mathbf{\bmu})$} is the normalizing constant.  
The pair $(\beta,\mathbf{Q})$ characterizes the dependence among the coordinates of $\bX$, and $X_i$'s are independent if $\beta\mathbf{Q}=\mathbf{0}_{n\times n}$.    
 We are interested in understanding the role of dependence $(\beta,\bQ)$ in testing against a collection of alternatives defined by a class of subsets $\mathcal{C}_n$ of $\{1,2,\ldots ,n\}$ each of which is of size $s$. More precisely, given any class of subsets $\mathcal{C}_n$ of $\{1,2,\ldots ,n\}$ of size $s\in [n]$, we consider testing the following hypotheses
\begin{equation} 
	H_0: \bmu=\mathbf{0} \quad {\rm vs} \quad H_1: \bmu \in \Xi(\mathcal{C}_n,s,A), \label{eqn:sparse_hypo}
\end{equation}
where
$${\Xi}(\mathcal{C}_n,s,A):=\left\{\begin{array}{c}\bmu\in \R_+^n: \mathrm{supp}(\bmu)\in \mathcal{C}_n,  \min\limits_{i\in {\rm supp}(\bmu)}\mu_i\geq A\end{array}\right\},\quad
\text{ and }{\rm supp}(\bmu):=\{i\in \{1,\ldots,n\}:\mu_i\ne 0\}.
$$
Thus the class of alternatives $\Xi(\mathcal{C}_n,s,A)$ puts non-zero signals on one of the candidate sets in $\mathcal{C}_n$ where each signal set has size $s$.  Throughout we shall assume that there exists a $\upsilon>0$ such that $s\leq n^{1-\upsilon}$. However, some of our results go through for $s$ as large as $\frac{n}{\log{n}}$. Finally, {\color{black} although we only consider one directional signals}, our results should go through for any non-critical $\beta$ for bi-directional signals as well.

Of primary interest here is to explore the effect of $(\beta,\mathbf{Q})$ on testing \eqref{eqn:sparse_hypo} for some structured signal classes $\mathcal{C}_n$. Examples of such signals will include geometric structures such as block signals on a lattice or suitable classes of low entropy signals (e.g. class of signals having enough disjoint sets -- see Section \ref{sec:mean_field} for precise definitions) on graphs with no inherent geometry. In this regard, previously, \cite{arias2011detection,arias2005near} studied the detection of block-sparse and thick shaped signals on lattices while \cite{addario2010combinatorial} considered general class of signals of combinatorial nature -- however both these papers assume independent outcomes which corresponds to $\beta=0$ in~\eqref{eqn:general_ising}. Several other papers have also considered detection of contiguous signals over lattices and networks (see e.g. \cite{enikeeva2018bump,zou2017nonparametric,arias2018distribution,sharpnack2015detecting,walther2010optimal,butucea2013detection,konig2020multidimensional} and references therein). However, in overwhelming majority of the literature, the networks in question have only been used to describe the nature of signals -- such as rectangles or thick clusters in lattices \citep{arias2011detection}. A fundamental question however remains -- ``how does dependence characterized by a network modulate the behavior of such detection problems?" In this regard, \cite{enikeeva2020bump} recently explored the effect of dependence on such detection problems for stationary Gaussian processes -- with examples including linear lattices studied through the lens of Gaussian auto-regressive observation schemes. Dependence structures beyond Gaussian random variables are often more challenging to analyze (due to possible lack of closed form expressions of resulting distributions) and allow for interesting and different behavior of such testing problems -- see e.g. \cite{mukherjee2016global}. One of the motivations of this paper is to fill this gap in the literature and show how dependent binary outcomes can substantially change the results for detecting certain classes structured signals. 

To this end, we adopt a standard asymptotic minimax framework as follows. Let a statistical test for $H_0$ versus $H_1$ be a measurable $\{0,1\}$ valued function of the data $\bX$, with $1$ denoting rejecting the null hypothesis $H_0$ and $0$ otherwise. The worst case risk of a test $T: \{\pm 1\}^{n}\to \{0,1\}$ for testing \eqref{eqn:sparse_hypo} is defined as 
\begin{align} 
	\mathrm{Risk}(T,{\Xi}(\mathcal{C}_n,s,A),\beta,\bQ)&:=\P_{\beta,\bQ,\mathbf{0}}\left(T(\bX)=1\right)+\sup_{\bmu \in {\Xi}(\mathcal{C}_n,s,A)}\P_{\beta,\bQ,\bmu}\left(T(\bX)=0\right). \label{eqn:general_hypo_ising}
\end{align}
We say that a sequence of tests $T_n$ corresponding to a sequence of model-problem pair (\eqref{eqn:general_ising}) and (\eqref{eqn:general_hypo_ising}), to be asymptotically powerful, asymptotically not powerful, and asymptotically powerless against $\Xi({\mathcal{C}_n,s,A})$  respectively, if
\begin{eqnarray*}
	\label{eqn:powerful}
	\limsup\limits_{n\rightarrow \infty}\mathrm{Risk}(T_n,\Xi(\mathcal{C}_n,s,A),\beta,\bQ)= 0,\\
		\liminf\limits_{n\rightarrow \infty}\mathrm{Risk}(T_n,\Xi(\mathcal{C}_n,s,A),\beta,\bQ)>0,\\
	\liminf\limits_{n\rightarrow \infty}\mathrm{Risk}(T_n,\Xi(\mathcal{C}_n,s,A),\beta,\bQ)=1.
\end{eqnarray*}
The goal of the current paper is to characterize how the sparsity $s$ and strength $A$ of the signal jointly determine if there is an asymptotically powerful test, and how the behavior changes with $(\beta,\bQ)$. In this regard, the main results of this paper are summarized below.

\begin{itemize}
\item{\bf General Upper Bounds}

\begin{enumerate}
\item [(I)] For a general class of $(\beta,\bQ)$ with $(\beta,\bQ)$ known, we show that a scan statistic can detect certain class of sparse signals \eqref{eqn:sparse_hypo} as soon as $\tanh(A)\gg \sqrt{\log n/s}$; see Theorem \ref{thm:upper} (in fact the change happens at a constant level). 

\item [(II)] A natural question is what happens if $(\beta,\bQ)$ are unknown. In this direction, we get the same detection boundary as above if the unknown $(\beta,\bQ)$ satisfies some additional assumptions (see Theorem \ref{thm:upper_unknown_beta_Q}).


\end{enumerate}

\item{\bf General Lower Bounds}
\begin{enumerate}
\item [(I)] If the signal set has \enquote{large} cardinality $s$, we provide a general lower bound by showing that no test is asymptotically powerful, under assumptions on correlations between spins (see Theorem \ref{thm:lower_short_range_large_s}) for Ferromagnetic Ising Models, if the signal $A$ is small.    


\item[(II)] The upper bound results suggest that testing is impossible if the signal set has \enquote{small} cardinality. We confirm this intuition by showing that {no test is asymptotically powerful irrespective of signal strength $A$} for small $s$, in Ferromagnetic Ising Models (see Theorem \ref{thm:lower_short_range_small_s}).


\end{enumerate}

\item{\bf Examples}

\begin{enumerate}
\item[(I)] {\bf Mean-Field Type Ising Models:} We apply our general results to several popular examples of Mean-Field Ising models. These include Ising models on dense regular graphs, random regular graphs with \enquote{large} degree, and  Erd\H{o}s-R\'{e}nyi graphs with \enquote{large} edge density. For $\beta\ne 1$ (which is the critical point for these Ising models), detection is impossible for small $s$ for any value of $A$, and detection is possible for large $s$ with the detection boundary $\tanh(A)\sim \sqrt{\frac{\log n}{s}}$. On the other hand, at criticality the detection boundary has three distinct regimes depending on the length of the signal set $s$, which we refer to informally as small, medium and large. For $s$ small, again no testing is possible for any signal strength $A$. For $s$ medium, detection is possible  with detection boundary $\tanh(A)\sim \sqrt{\frac{\log n}{s}}$, which is the same as the case $\beta\ne 1$. Finally if $s$ is large, the detection boundary shifts to $\tanh(A)\sim \frac{n^{1/4}}{s}$ instead, and thus allows for detection of much smaller signals only for $\beta=1$. This improved upper bound at criticality for $s$ large does not follow from Theorem~\ref{thm:upper} (which is based on a scan test), but instead utilizes a test based on sum of spins.
The proof of the lower bound requires bounds on correlation between spins at all temperatures (see e.g. Lemma \ref{lem:slarge} and other supporting results in Section \ref{sec:comres}). The proof of the upper bound at criticality follows from a careful analysis the sum of spins (see Lemma \ref{lem:altbeh}).
To the best of our knowledge, these results are new, and might be of independent interest.

\item [(II)] {\bf Ising Models on Lattices:} We show that for the classical Ising model on any fixed $d$-dimensional lattice the detection boundary again scales like $\tanh(A)\sim \sqrt{\frac{\log n}{s}}$ throughout the high temperature regime (right up to the critical temperature). The proof uses finite volume correlation decay, and ratio-scale mixing results. We note that similar arguments should apply in the low temperature positive pure-phase regime (plus boundary conditions). The case of free boundary conditions in the low temperature regime remains open.
\end{enumerate}
\end{itemize}

\subsection{Future scope} In this paper we have explored how the level of dependence in Ising models can modulate the behavior of detection problems for testing certain structured anomalies. This minimax hypothesis testing problem provides a natural next step in a rich area of research of detecting contiguous signals of geometric nature -- yet mostly under independence of the outcomes. Although we pinpoint the rates of minimax separation in this paper, we believe that there is a sharp constant threshold at which the transition happens from test-ability to non-test-ability  (see e.g. \cite{arias2011detection} for independent outcomes). In a different direction, one can study whether it is possible to attain optimal rates of detection for all $\beta$, if $\beta$ is unknown (but $\mathcal{Q}$ is known, say).


\subsection{Organization} The rest of the paper is organized as follows.  In Section \ref{section:upper} we present some general upper bounds -- including both the case of known and unknown dependence parameters $(\beta,\bQ)$. Section \ref{section:lower} contains a general lower bound result under Ferromagnetic condition (positive $\beta,\bQ$) and correlation decay type conditions. Subsequently, in Section \ref{section:Examples} we apply these general results to demonstrate sharp upper and lower bounds for detecting signals in several commonly studied classes of Ising Models. Section \ref{section:proofs} contains the proofs of the main results from Sections~\ref{section:upper}~and~\ref{section:lower}. In Section~\ref{sec:corpf}, we present the proofs of the theorems stated in Section~\ref{section:Examples}. Section~\ref{sec:comres} contains proofs of additional technical lemmas which may be of independent interest, pertaining to bounds on mixing, spin-spin correlations and asymptotic analysis of the sum of spins at critical temperature.

\subsection{Notation}\label{section:notation} 
Throughout,  $\E_{\beta,\bQ,\bmu},\mathrm{Var}_{\beta,\bQ,\bmu},\mathrm{Cov}_{\beta,\bQ,\bmu}$ will denote the expectation, variance, and covariance operators corresponding to the measure $\P_{\beta,\bQ,\bmu}$. For a vector $\bX$ following the Ising model as in~\eqref{eqn:general_ising}, the $i^{th}$ coordinate will be denoted by $X_i$, and  will often be referred to as the spin for vertex $i$. For a given sequence of symmetric matrices $\mathcal{Q}=\{\bQ_{n\times n}\}_{n\ge 2}$ with non-negative entries, we define the critical temperature as
\begin{align}\label{eq:critical}
\beta_{c}(\mathcal{Q})=\inf\left\{\beta>0: \lim_{h\downarrow 0}\lim_{n\rightarrow \infty}\E_{\beta,\bQ,\bmu(h)}\left(\frac{1}{n}\sum_{i=1}^n X_i\right)>0\right\},
\end{align}
where $\bmu(h)=(h,\ldots,h)^T\in \mathbb{R}^n$ denotes the vector with all coordinates equal to $h$. Similarly, for any $S\in \mathcal{C}_n$ and real number $\eta$ let $\bmu_S(\eta)$ denote the vector which has $\mu_i=\eta\I(i\in S)$.  
In the examples pursued in the rest of the paper the existence of the limit is a part of classical statistical physics literature, and we shall note relevant references whenever talking about critical temperature in our examples. We shall refer to $(0,\beta_c(\mathcal{Q}))$ as high temperature regime, and $\beta\in (\beta_c(\mathcal{Q}),\infty)$  as low temperature regime. 

For any $a,b \in \mathbb{N}$, we let $[a:b]=\{a,a+1,\ldots,b\}$ and $[a]=\{1,\ldots,a\}$. We also denote the $m$-dimensional $1$-vector $(1,1,\ldots,1)\in \mathbb{R}^m$ by $\mathbf{1}_m$. 
Also for any finite set $S$ we use $|S|$ to denote the number of elements in $S$. For any two vectors $\mathbf{v}_1,\mathbf{v}_2$ of same dimension and $1\leq p\leq \infty$ we let $\|\mathbf{v}_1-\mathbf{v}_2\|_p$ denote the Euclidean $\ell_p$ norm. For any real matrix $\mathbf{M}$ and $1\leq p\leq \infty$ we define the $p$-matrix norm of $\mathbf{M}$ as $\|\mathbf{M}\|_{p\rightarrow p}=\sup_{\|\mathbf{v}\|_p=1}\|\mathbf{M}\mathbf{v}\|_p$. For $p=2$, we drop the the subscript to use $\|\mathbf{M}\|$ as the spectral norm of $\mathbf{M}$. For $p=\infty$ we use the fact that $\|\bQ\|_{\infty\rightarrow\infty}=\sup_{i\in[n] }\sum_{j\in [n]}|\bQ_{ij}|$. 
Finally for any vector $\mathbf{v}\in \mathbb{R}^m$ and subset $S\subset [m]$ we let $\mathbf{v}_S$ to be the $|S|$-dimensional vector obtained by restricting $\mathbf{v}$ to the coordinates in $S$. For $\bX\sim \P_{\beta,\bQ,\bmu}$ as in~\eqref{eqn:general_ising}, let $m_i=m_i(\bX):=\sum_{j=1}^n \bQ_{ij}X_j$.

The results in this paper are mostly asymptotic (in $n$) in nature and thus requires some standard asymptotic  notations.  If $a_n$ and $b_n$ are two sequences of real numbers then $a_n \ll b_n$ and $a_n=o(b_n)$ implies that ${a_n}/{b_n} \rightarrow  0$ as $n \rightarrow \infty$. Similarly  $a_n \lesssim b_n$ and $a_n=O(b_n)$  implies that  $\limsup_{n \rightarrow \infty} {{a_n}/{b_n}} <\infty$.
We also say $a_n=\Theta(b_n)$ or $a_n \sim b_n$ if both $a_n=O(b_n)$ and $b_n=O(a_n)$. 
Finally for any integer $m\geq 1$ and real $0\leq p\leq 1$ we let $\mathrm{Bin}(m,p)$ denote the distribution of a binomial distribution with $m$ trials and success probability $p$. 

\section{General results} We divide our general results into two subsections pertaining to upper and lower bounds.

\subsection{\bf Upper Bounds: General Coupling Matrix}\label{section:upper}

Let us assume that the class of signals $\mathcal{C}_n$ satisfies
\begin{equation}\label{eq:forupper}
\log{|\mathcal{C}_n|}\leq C_u\log{n}, \quad |\mathcal{C}_n|\to\infty,
\end{equation}
for some constant $C_u>0$.
We now begin with a general result which pertains to pinning down a signal strength necessary for detection in a general class of $\bQ$. To describe the test, define for any $S\in \mathcal{C}_n$
\begin{align} 
L_{S}(\bmu):=\frac{1}{\sqrt{|S|}}\sum_{i\in S} (X_i-\tanh(\beta m_i+\mu_i))\nonumber,
\end{align}
where $m_i=\sum_{j=1}^n \bQ_{ij}X_j.$
For a fixed {\color{black} $\delta\in (0,1)$}, 
consider the test rejects the null hypothesis when 
\begin{align*}
L_n:=\sup_{S\in \mathcal{C}_n}|L_{S}(\mathbf{0})|>2(1+\beta\|\bQ\|_{\infty \rightarrow \infty})\sqrt{2(1+\delta)\log{|\mathcal{C}_n|}},
\end{align*}

\begin{theorem}\label{thm:upper}
	Suppose $\bX\sim \P_{\beta,\bQ,\bmu}$ with any $\beta\in \mathbb{R}$ and $\mathbf{Q}$ such that $\|\bQ\|_{\infty\to \infty}\leq C_u'$ for some constant $C_u'$ and~\eqref{eq:forupper} holds. Consider testing hypotheses about $\bmu$ as described by \eqref{eqn:sparse_hypo}. Then there exists a constant $C'>0$ such that if $\tanh(A)\ge  C'\sqrt{\frac{\log{n}}{s}}$, then the test based on $L_n$ defined above is asymptotically powerful.
\end{theorem}

 The fact that the test that attains the performance claimed in Theorem \ref{thm:upper} is, not surprisingly, a scan type procedure. However, to attain optimal separation rates across all regimes of dependence (i.e. $\beta$) we need to conditionally center the scanning elements instead of unconditional centering prevalent for hypothesis testing literature with independent outcomes. 
This version of the scan test relies explicitly on the full knowledge of the null distribution. Especially this requires that $\beta,\bQ$ are known. As we shall show in Section \ref{section:Examples}, this test is indeed optimal for any non-critical $\beta>0$ for a large class of underlying graphs. If however, the values of $(\beta,\bQ)$ are unknown but \enquote{small}, there exists a sequence of tests with the same detection thresholds as above, which does not depend on the knowledge of $\beta,\bQ$. 

To define the test statistic, for any $S\in \mathcal{C}_n$ set
\begin{align} 
\tilde{L}_{S}(\bmu):=\frac{1}{\sqrt{|S|}}\sum_{i \in S} (X_i-\EG X_i).\nonumber
\end{align}
Fixing $\delta\in (0,1)$, consider the test which rejects the null hypothesis when
\begin{align*}
\tilde{L}_n:=\sup_{S\in \mathcal{C}_n}|\tilde{L}_{S}(\mathbf{0})|>\sqrt{\frac{(1+\delta)\log |\mathcal{C}_n|}{1-\eta}}.
\end{align*}

\begin{theorem}\label{thm:upper_unknown_beta_Q}
	Suppose $\bX\sim \P_{\beta,\bQ,\bmu}$ with $0\leq \beta\|\bQ\|_{\infty\rightarrow \infty}\le \eta$ for some $0\leq \eta<1$ fixed, $\min_{i,j}\beta\bQ_{i,j}\geq 0$ and~\eqref{eq:forupper} holds. Consider testing hypotheses about $\bmu$ as described by \eqref{eqn:sparse_hypo}. Then for $C'>0$ large enough,  if $\tanh(A)\ge  C'\sqrt{\frac{\log{n}}{s}}$, the test based on $\tilde{L}_n$ above is asymptotically  powerful.
\end{theorem}

\begin{remark}
We note that the condition $\eta<1$ is sharp, in that the test proposed in Theorem \ref{thm:upper_unknown_beta_Q} does not control type I error as soon as $\eta=1$. A counter example is provided by the Curie Weiss model, (Ising model on the complete graph) for which $\eta=1$ allows for the choice $\beta=\beta_c=1$, which is the critical temperature for this model (see the discussions after Theorem \ref{thm:dense_regular} in Section \ref{section:Examples} for exact details). 
\end{remark}

Examples of $\bQ$ which satisfy the assumption stated in Theorem \ref{thm:upper} and Theorem \ref{thm:upper_unknown_beta_Q} include some prototypical examples of Ising models studied in literature. We discuss them in detail in Section \ref{section:Examples}.


\subsection{\bf Lower Bounds: Ferromagnetic Models}\label{section:lower}
In this section we present results on lower bounds to demonstrate sharpness of Theorem \ref{thm:upper} for Ising models having $\min _{i,j}\beta \bQ_{ij}\geq 0$ -- traditionally referred to as Ferromagnetic model. 
In this regard, according to Theorem \ref{thm:upper}, successful detection is possible by a conditionally centered scan test provided $\tanh(A)\geq C' \sqrt{\frac{\log{n}}{s}}$ for a constant $C'>0$ which depends on the class of signals $\mathcal{C}_n$ and $\|\bQ\|_{\infty\rightarrow \infty}$ through the constants $C_u$ and $C_u'$. Since $\tanh(A)\in (-1,1)$, it seems that one might need $s$ to at least be of order $\log{n}$ for the success of this test. This intuition turns out to be true and there exists a phase transition in the possibility of testing depending on the behavior of $s$ w.r.t $\log{n}$. In particular, there exists constants $0< c\leq C<\infty$ (depending on the problem sequence $(\beta,\bQ)$ and class of alternatives $\mathcal{C}_n$) such that the detection problem behaves differently depending on whether $s\leq c\log{n}$ or $s\geq C\log{n}$. Before formally stating the relevant results, let us assume that there exists a constant $C_l>0$ and some sub-collection $\mathcal{C}_n'\subseteq \mathcal{C}_n$ of disjoint sets such that
\begin{equation}\label{eq:forlower}
\log{|\mathcal{C}_n'|}\geq C_l\log{n}.
\end{equation}
Note that \eqref{eq:forlower} immediately implies $\min(|\mathcal{C}_n|,|\mathcal{C}_n'|)\to \infty$, and so it need not be assumed separately. As the proofs of the results in the two regimes ($s$ small/large) involve substantially different ideas, we divide their presentation in separate subsections.

\subsubsection{Large signal size $s$ }  

The following theorem will be used to verify sharpness of the upper bound presented in Theorem \ref{thm:upper} for $s$ large.

\begin{theorem}\label{thm:lower_short_range_large_s}
Suppose $\bX\sim \P_{\beta,\bQ,\bmu}$ such that $\min\limits_{i,j}\beta\bQ_{ij}\geq 0$, $\|\bQ\|_{\infty\rightarrow \infty}\leq C_u'$ for some constant $C_u'>0$ and~\eqref{eq:forlower} holds. Consider testing \eqref{eqn:sparse_hypo}. Then there exists fixed constants $c',C>0$ such that the following hold:

\begin{enumerate}
	\item [(I)] Suppose there exists sequences $r_{n},r_{n}'$ diverging to $+\infty$ with  $r_n\geq C\log{n}$, and 
	\begin{itemize}
		\item $\sup\limits_{S\in \mathcal{C}_n'}\mathrm{Var}_{\beta,\bQ,\mathbf{0}}\left( \sum_{i\in S}X_i\right)\leq r_n $.
		
			\item $\sup\limits_{S_1\neq S_2\in \mathcal{C}_n'} \mathrm{Cov}_{\beta,\bQ,\mathbf{0}}\left(\sum_{i\in S_1} X_i,\sum_{j\in S_2} X_j\right)=o( r_n').$
		
		%
	\end{itemize}

Then all tests are asymptotically powerless when $\tanh(A)\leq c' \min\left\{\sqrt{\frac{\log{n}}{r_n}},\sqrt{\frac{1}{r_n'}}\right\}$. 

%
%

\item [(II)] 
Suppose there exists 
 an increasing set $\Omega_n\subseteq \{-1,+1\}^n$, constant $\kappa>0$, sequences $r_n,r_n'$ diverging to $+\infty$ with $r_n\geq C\log{n}$, such that: 

\begin{itemize}
	

	%
	\item $\sup_{\eta\in [0.A]}\sup_{S\in\mathcal{C}_n'}\mathrm{Var}_{\beta,\bQ,\bmu_S(2\eta)}(\sum_{i\in S}X_i|\Omega_n)\leq r_n.$
	\item $\sup_{\eta\in [0,A]}\sup_{S_1\neq S_2\in\mathcal{C}_n'}\big|\mathrm{Var}_{\beta,\bQ,\mathbf{\bmu}_{S_1\cup S_2}(\eta)}\left(\sum_{i\in S_1\cup S_2}X_i|\Omega_n\right)-\mathrm{Var}_{\beta,\bQ,\bmu_{S_1}(\eta)}\left(\sum_{i\in S_1}X_i|\Omega_n\right)-\mathrm{Var}_{\beta,\bQ,\bmu_{S_2}(\eta)}\left(\sum_{i\in S_2}X_i|\Omega_n\right)\big|= o(r_n'),$
			\item $\liminf_{n\rightarrow \infty}\P_{\beta,\bQ,\mathbf{0}}(\Omega_n)\geq \kappa$.
\end{itemize}
Then no test is asymptotically powerful if $\tanh(A)\leq c' \min\left\{\sqrt{\frac{\log{n}}{r_n}},\sqrt{\frac{1}{r_n'}}\right\}$.

\end{enumerate}

\end{theorem}

A few remarks are in order about the assumptions and implications of the result. First, the results are presented in two parts since they will eventually be applied (in Section \ref{section:Examples}) to prove sharp lower bounds for high and low temperature regimes separately. Indeed, for $\beta<\beta_c$ (for $\beta_c$ the critical temperature in specific examples) we will appeal to part I of the theorem while part II of the theorem will be used for the low temperature regime $\beta>\beta_c$ with the increasing event $\Omega_n$ typically being $\bar{\bX}\geq 0$. Moreover, in most of our examples $r_n =\Theta(s)$ and consequently, the condition $r_n\geq C\log{n}$ in Theorem~\ref{thm:lower_short_range_large_s} is equivalent to $s\geq C\log{n}$ for a (potentially) different constant $C>0$. The other term in the lower bound (corresponding to $r_n'$) plays a crucial role only at the critical temperature $\beta=\beta_c$. Finally, the main quantity that decides the validity of the lower bound presented above happens to be the correlation between spins $X_i$ and $X_j$ for suitable pairs $i,j$. Indeed, such correlation control is an area of active research and eventual verification of these conditions requires establishing correlation bounds on the graphs in our examples. We derive several new such bounds in Section \ref{sec:comres}.


\subsubsection{Small signal size $s$}

The following theorem will be used to verify sharpness of the upper bound presented in Theorem \ref{thm:upper} for $s$ small.

\begin{theorem}\label{thm:lower_short_range_small_s}
Suppose $\bX\sim \P_{\beta,\bQ,\bmu}$ such that $\min\limits_{i,j}\beta\bQ_{ij}\geq 0$ and~\eqref{eq:forlower} holds. Consider testing \eqref{eqn:sparse_hypo}. Then there exists $c>0$ such that whenever $s\leq c\log{n}$, the following holds:

\begin{enumerate}
\item[(I)]
If 
\[\lim_{n\to\infty}\sup_{S_1\ne S_2\in \mathcal{C}_n'}\Big|\frac{\P_{\beta,\bQ,\mathbf{0}}(\bX_{S_1}=1,\bX_{S_2}=1)}{\P_{\beta,\bQ,\mathbf{0}}(\bX_{S_1}=1)\P_{\beta,\bQ,\mathbf{0}}(\bX_{S_2}=1)}-1\Big|=0,\]
then all tests are asymptotically powerless irrespective of $A$.

\item[(II)]
If there exists an increasing set $\Omega_n$, a constant $\kappa>0$ such that $\liminf_{n\rightarrow\infty}\P(\Omega_n)>\kappa$ and the following holds:

 \[\lim_{n\to\infty}\sup_{S_1\ne S_2\in \mathcal{C}_n'}\Big|\frac{\P_{\beta,\bQ,\mathbf{0}}(\bX_{S_1}=1,\bX_{S_2}=1|\Omega_n)}{\P_{\beta,\bQ,\mathbf{0}}(\bX_{S_1}=1|\Omega_n)\P_{\beta,\bQ,\mathbf{0}}(\bX_{S_2}=1|\Omega_n)}-1\Big|=0,\]
 then no test is asymptotically powerful irrespective of $A$.

\end{enumerate}

\end{theorem}
Theorem \ref{thm:lower_short_range_small_s} verifies our intuition from Theorem \ref{thm:upper} {\color{black} that no signal can be detected for $s\leq c \log{n}$ for suitably small $c>0$.}  The verification of the conditions on Theorem \ref{thm:lower_short_range_small_s} essentially reduces to establishing versions of ratio scale strong-mixing for Ising models.

\section{\bf Specific Examples}\label{section:Examples} 
In this section we provide examples of coupling strengths $(\beta,\bQ)$ and class of signals $\mathcal{C}_n$ which when looked through the lens of Theorems \ref{thm:upper},  \ref{thm:lower_short_range_large_s}, and  \ref{thm:lower_short_range_small_s} yield matching (in terms of rate) upper and lower bounds. 

\subsection{\bf Mean Field Interactions}\label{sec:mean_field}

In this section we verify the validity and optimality of our upper and lower bounds for some examples of Mean-Field type Ising models. Mean-Field Ising models can be generally characterized by positing conditions on $\bQ$ under which mean-field approximation holds (see e.g. \cite{basak2017universality,jain2018mean} for exact definitions and details). Our results on rate optimal detection boundaries will be verified for some important sub-classes of such mean-field Ising models.  

In this section, throughout, the  matrix $\mathbf{Q}$ will usually be associated with a certain sequence of simple labeled graphs $\mathbb{G}_n=(\mathcal{V}_n,\mathcal{E}_n)$ with vertex set $\mathcal{V}_n=\{1,\dots,n\}$ and edge set $\mathcal{E}_n \subseteq \mathcal{V}_n\times \mathcal{V}_n$ and corresponding $\mathbf{Q}= \mathbf{G}_n/\overline{d}$ where we define the average degree of the graph $\mathbb{G}_n$ to be $\overline{d}=|\mathcal{E}_n|/|\mathcal{V}_n|$. 
Here $\mathbf{G}_n$ is the adjacency matrix of $\mathbb{G}_n$. Also given any square matrix $\mathbf{M}$, we use $\lambda_i(\mathbf{M})$ to denote the $i$-th largest eigenvalue of $\mathbf{M}$.

Since such models have no apparent geometry, there is less restriction on the choice of signal classes $\mathcal{C}_n$. Consequently, in this section we discuss testing against sparse alternatives of size $s$ define by $\Xi(\mathcal{C}_n,s,A)$, where $\mathcal{C}_n$ is any collection of subsets of $\mathcal{V}_n$ of size $s$ such that~\eqref{eq:forupper}~and~\eqref{eq:forlower} hold.

In all the examples to be considered in this subsection, the critical temperature corresponds to $\beta_c=1$ (see e.g. \cite{basak2017universality}), and we demonstrate a double phase transition on the detection boundary in terms of the signal size $s$ at the critical temperature (compared to only one phase transition at $s\sim \log{n}$ for non-critical temperatures), at $s\sim \log{n}$ and $s\sim \sqrt{n}/\log{n}$. In particular, for $s\gtrsim \sqrt{n}/\log{n}$ the behavior of the testing problem changes at the critical temperature, and one is able to detect lower signals using a simple test based on total number of spins.
\\

{\color{black} 

Our first example in this regard is for dense regular graphs. 
 
\begin{theorem}\label{thm:dense_regular}
Suppose $\mathbb{G}_n$ corresponds to a $d_n$-regular graph, which is dense, i.e.  $d_n=\Theta(n)$ and~\eqref{eq:forupper}~and~\eqref{eq:forlower} hold. Then there exists constants $c,C>0$ such that if,
\begin{enumerate}
\item[(a)]
$s\le c\log n$, then the following conclusions hold:
\begin{itemize}
\item
If $\beta\le 1$, all tests are asymptotically powerless for any $A>0$. 

%
\item
If $\beta>1$, no test is asymptotically powerful for any $A>0$, provided $\limsup_{n\rightarrow\infty}\frac{\lambda_2(\mathbf{G}_n)}{d_n}<1$. 
\end{itemize}

\item[(b)]
$s\ge C\log n$, then there exists constants $c', C'>0$ such that the following conclusions hold:
\begin{itemize}
\item
If $\beta<1$, all tests are asymptotically powerless if $\tanh(A)\le c'\sqrt{\frac{\log n}{s}}$. On the other hand, if $\tanh(A)\ge C'\sqrt{\frac{\log n}{s}}$, there is a sequence of asymptotically powerful tests.

\item
Suppose $\beta=1$ and $\limsup\limits_{n\to\infty} \lambda_2(\mathbf{G}_n)/d_n<1$. 

\begin{itemize}
    \item 
If $s\ll \sqrt{n}/\log{n}$, all tests are asymptotically powerless if $\tanh(A)\le c'\sqrt{\frac{\log n}{s}}$, and there is a sequence of asymptotically powerful tests 
if $\tanh(A)\ge C'\sqrt{\frac{\log n}{s}}$. 

\item 
If $s\gtrsim \sqrt{n}/\log{n}$, all tests are asymptotically powerless if $s\tanh(A)\ll n^{1/4}$ and there exists a sequence of asymptotically powerful tests if $s\tanh(A)\gg n^{1/4}$.
\end{itemize}

\item
If $\beta>1$ and $\limsup\limits_{n\to\infty} \lambda_2(\mathbf{G}_n)/d_n<1$, there are no asymptotically powerful tests if $\tanh(A)\le c'\sqrt{\frac{\log n}{s}}$. On the other hand, if $\tanh(A)\ge C'\sqrt{\frac{\log n}{s}}$, there is a sequence of asymptotically powerful tests.

\end{itemize}

\end{enumerate}

\end{theorem}

Theorem \ref{thm:dense_regular}, which includes the classical Curie-Weiss Model as a special case (corresponding to $d_n=n-1$), demonstrates the benefit of critical temperature in detecting lower signals for $s\gtrsim \sqrt{n}/\log{n}$. In particular, for any $\beta\neq \beta_c$, the detection thresholds resemble that of $\beta=0$ (i.e. independent observations) and only for $\beta=\beta_c$ one can detect lower signals -- and that too only when the number of signals is large enough. 

The assumption of denseness of the regular graph can be removed under randomness. In particular, a similar result holds for sparser but random regular graphs. We state this in our next result.

\begin{theorem}\label{thm:random_regular}
Suppose $\mathbb{G}_n$ is the adjacency matrix of a $d_n$ random regular graph, with 
\[\theta:=\liminf_{n\rightarrow\infty}\frac{\log d_n}{\log n}\in [0,1]\]
and~\eqref{eq:forupper},~\eqref{eq:forlower} hold. Then there exists fixed constants $c,C>0$ such that if,
\begin{enumerate}
\item[(a)]
$s\le c\log n$, then the following conclusions hold:
\begin{itemize}
\item
If $\beta\le 1$, all tests are asymptotically powerless for any $A$, provided $\theta>0$.

%
\item
If $\beta>1$, no test is asymptotically powerful for any $A$, provided $\theta>0$.
\end{itemize}

\item[(b)]
$s\ge C\log n$, then there exists constants $c',C'>0$ such that the following conclusions hold:
\begin{itemize}
\item
If $\beta<1$, all tests are asymptotically powerless if $\tanh(A)\le c'\sqrt{\frac{\log n}{s}}$ and $\theta>1/2$. On the other hand, if $\tanh(A)\ge C'\sqrt{\frac{\log n}{s}}$, there is a sequence of asymptotically powerful tests for any $\theta \geq 0$.

\item
Suppose $\beta=1$. 

\begin{itemize}
    \item 
If $s\ll \sqrt{n}/\log{n}$, all tests are asymptotically powerless if $\tanh(A)\le c'\sqrt{\frac{\log n}{s}}$, $\theta>1/2$ and there is a sequence of asymptotically powerful tests if $\tanh(A)\ge C'\sqrt{\frac{\log n}{s}}$ and $\theta\geq 0$. 

\item 
If $s\gtrsim \sqrt{n}/\log{n}$ and $\theta>1/2$, then all tests are asymptotically powerless if $s\tanh(A)\ll n^{1/4}$, and there exists a sequence of asymptotically powerful tests if $s\tanh(A)\gg n^{1/4}$.
\end{itemize}

\item
If $\beta>1$, there are no asymptotically powerful tests if $\tanh(A)\le c'\sqrt{\frac{\log n}{s}}$ and $\theta>2/3$. On the other hand, if $\tanh(A)\ge C'\sqrt{\frac{\log n}{s}}$, there is a sequence of asymptotically powerful tests for any $\theta \geq 0$.

\end{itemize}

\end{enumerate}

\end{theorem}

It is intuitive that the results for random regular graphs should naturally extend to suitable Erd\H{o}s-R\'{e}nyi graphs as well. This intuition is indeed correct -- as verified by our next result.

\begin{theorem}\label{thm:erdos_renyi}
Suppose $\mathbb{G}_n$ is the adjacency matrix of an Erd\H{o}s-R\'{e}nyi random graph with parameter $p_n$, such that
\[\theta:=\liminf_{n\rightarrow\infty}\frac{\log (np_n)}{\log n}\in [0,1]\]
as before and~\eqref{eq:forupper},~\eqref{eq:forlower} hold. Then the same conclusions hold as in Theorem~\ref{thm:random_regular} except that every occurrence of the condition $\theta\geq 0$ in part (b) of Theorem~\ref{thm:random_regular} is replaced with $\theta>0$.

\end{theorem}

A summary of the detection boundary for mean field Ising models is given in the tree in Figure \ref{fig:SummaryTree}. Even though all the transitions happen at a constant level, we remove all constants to make the results more transparent.

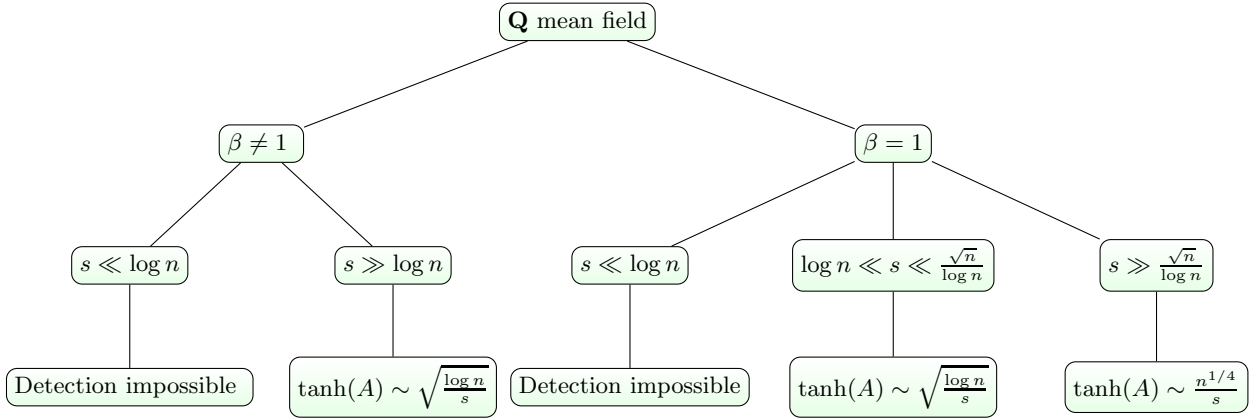
\begin{figure}
\begin{center}
\begin{tikzpicture}[level distance=1.62cm,
			level 1/.style={sibling distance=8.4cm},
			level 2/.style={sibling distance=3.5cm},
  every node/.style = {shape=rectangle, rounded corners,
    draw, align=center,
    top color=white, bottom color=green!10}]]
  \node {$\bQ$ mean field}
    child { node {$\beta\ne 1$ } 
      child{ node{$s\ll \log n$} child {node {Detection impossible  }}}
    child { node {$s\gg \log n$}child{node{ $\tanh(A)\sim \sqrt{\frac{\log n}{s}}$}}} } child{ node {$\beta=1$} child{node{$s\ll \log n$} child{node{Detection impossible}}} child{node{$\log n\ll s\ll \frac{\sqrt{n}}{\log n}$} child{node { $\tanh(A)\sim \sqrt{\frac{\log n}{s}}$}}}child{node{$s \gg \frac{\sqrt{n}}{\log n}$} child{node{$\tanh(A)\sim \frac{n^{1/4}}{s}$}}}};
\end{tikzpicture}
\caption{Summary of detection boundary in Mean field models.}
\label{fig:SummaryTree}
\end{center}
\end{figure} 
The proofs of the theorems above, mostly relies on verifying the conditions of Theorems \ref{thm:upper},  \ref{thm:lower_short_range_large_s}, and  \ref{thm:lower_short_range_small_s} for the respective graphs. Only for $\beta=\beta_c=1$ and $s\gtrsim \sqrt{n}/\log{n}$, the optimal upper bound does not follow from Theorem \ref{thm:upper}. In this case the optimal test is not based on a scan test but rather simply on the total magnetization $\sum_{i=1}^n X_i$. The sharp analysis of the test based $\sum_{i=1}^n X_i$ requires several additional technical details. We develop the necessary ingredients in Section \ref{sec:comres}.

\subsection{\bf Short Range Interactions}\label{section:lattice}
 Indeed, the most classical example of an Ising Model corresponds to nearest neighbor interactions on a lattice in dimension $d$ \citep{ising1925beitrag,onsager1944crystal}. To introduce this model it is convenient to rewrite the vertices of the graph as the vertices of a lattice as follows. Given positive integer $d$, consider a growing sequence of integer lattice hypercubes of dimension $d$ defined by  $\Lambda_n(d):=[-n^{1/d},n^{1/d}]^d\cap \mathbb{Z}^d$, $n\geq 1$, where $\mathbb{Z}^d$ denotes the d-dimensional integer lattice. For any two distinct elements $i,j\in \Lambda_n(d)$ we put a weight $\mathbf{Q}_{ij}$, with the restriction that $\mathbf{Q}_{ij}=\mathbf{Q}_{ji}$. Thus $\mathbf{Q}$ is a symmetric array with zeros on the diagonal. With this notation, we say $\bQ$ is short range if there exists $L\geq 1$ such that $\bQ_{ij}=\bQ_{ij}(\Lambda_n(d),L)=\I(0<\|i-j\|_1\leq L)$. Since such a model has an inherent geometry given by the lattic structure in $d$-dimensions, it is natural to consider signals which can be described by such geometry. Similar to one of the emblematic cases considered in \cite{arias2005near,arias2011detection,walther2010optimal,butucea2013detection,konig2020multidimensional}, here we discuss testing against block sparse alternatives of size $s$ define by $\Xi(\mathcal{C}_n,s,A)$ with
 \begin{align}
\mathcal{C}_n=\left\{\prod\limits_{j=1}^d[a_j:b_j]\cap \Lambda_n(d): \ b_j-a_j=\lceil s^{1/d}\rceil\right\}.\label{eqn:cube_signals}
 \end{align}
 Although we only present the results for sub-cube detection in this paper, one can easily extend the results to detection of thick clusters (see \cite{arias2011detection} for details) with minor modifications of the arguments presented here.

	We now argue that the conditions of Theorems \ref{thm:lower_short_range_large_s} and \ref{thm:lower_short_range_small_s} hold right up to the critical temperature in such model and class of signals problem pair and therefore we have sharp matching lower bounds corresponding to the upper bounds presented after Theorem \ref{thm:upper}. In the following analyses, we let $\beta_c(d,L)$ denote the critical temperature of an Ising model with $\mathcal{Q}=\bQ(\Lambda_{n}(d),L)$ in \eqref{eq:critical}. Although analytic forms of $\beta_c(d,L)$ are intractable for $d\geq 3$, the existence of such critical temperatures has been classically studied -- see e.g. \cite{Ellis_Newman,duminil2017lectures,friedli2017statistical} for more details.
	\begin{theorem}\label{thm:lattice}
	Let $\bX \sim \P_{\beta,\bQ,\bmu}$ with  $\bQ_{ij}=\I(0<\|i-j\|_1\leq L)$ for $i,j\in \Lambda_n(d)$ and $0\leq \beta<\beta_c(d,L)$,  and consider testing \eqref{eqn:sparse_hypo} with $\mathcal{C}_n$ as in 
	\eqref{eqn:cube_signals}. Then there exists  positive constants 
	$c,C>0$ depending on $\beta,L,d$ such that the following hold.
	\begin{enumerate}
		\item[(a)] Suppose $s\leq c\log{n}$. Then all tests are asymptotically powerless irrespective of $A$.
		\item[(b)] Suppose $s\geq C\log{n}$. Then there exists constants $c',C'>0$ such that if
		 $\tanh(A)\ge C' \sqrt{\frac{\log{n}}{s}}$, there exists a sequence of asymptotically powerful tests. On the other hand, if $\tanh(A)\le c' \sqrt{\frac{\log{n}}{s}}$, then all tests are asymptotically powerless.		
	\end{enumerate}
	\end{theorem}

We note that at the critical point $\beta=\beta_c(d,L)$ we do not expect Theorem~\ref{thm:lattice} to hold, and the detection boundary to be lower (see the discussion in \cite{mukherjee2019testing} for heuristics in this regard). 

	We conclude this example by considering the case of one dimensional Ising Model i.e. $d=1$. This is the earliest studied Ising model and has $\bQ$ correspond to the adjacency matrix of the line graph on $n$ vertices \citep{ising1925beitrag}. 
	It is well known that the Ising model on the line graph does not exhibit a thermodynamic phase transition i.e. $\beta_c(1,1)=+\infty$ (see e.g. \cite{ising1925beitrag} and  \cite[Section 3.3]{friedli2017statistical}). As an immediate corollary to Theorem \ref{thm:lower_short_range_large_s} and Theorem \ref{thm:lower_short_range_small_s}, we get that the detection boundary remains the same for any $\beta\ge 0$, and is the same as the independent case i.e. $\beta=0$.	
	\begin{corollary}\label{corr:linegraph}
		Let $\bX \sim \P_{\beta,\bQ,\bmu}$ with $d=1$, $\bQ_{ij}=\I(|i-j|=1)$ for $i,j\in \Lambda_n(d)$ and $\beta\in \mathbb{R}^+$,  and consider testing \eqref{eqn:sparse_hypo} with $\mathcal{C}$ as in 
	\eqref{eqn:cube_signals}. Then there exists  positive constants 
	$c,C>0$ depending on $\beta$ such that the following hold.
	\begin{enumerate}
		\item[(a)] Suppose $s\leq c\log{n}$. Then all tests are asymptotically powerless irrespective of $A$.
		
		\item[(b)] Suppose $s\geq C\log{n}$. Then there exists constants $c',C'>0$ such that if
		 $\tanh(A)\ge C' \sqrt{\frac{\log{n}}{s}}$, there exists a sequence of asymptotically powerful tests. On the other hand, if $\tanh(A)\le c' \sqrt{\frac{\log{n}}{s}}$, then all tests are asymptotically powerless.
		
	\end{enumerate}
	\end{corollary}
We note that for this particular case, a different proof using exact expressions for the log partition function 
can be used to show the validity of Corollary  \ref{corr:linegraph} for any $\beta\in \mathbb{R}$. 

\section{Proofs of Main Results}\label{section:proofs}

\subsection{Some Supporting Lemmas}\label{section:supporting_lemmas}

In this section we collect the lemmas (whose proofs we defer to Section \ref{sec:pfsecsupplem}) which will be used in the proofs of Theorems~\ref{thm:upper}--\ref{thm:lower_short_range_small_s}. 

\begin{lemma}[GHS Inequality \citep{lebowitz1974ghs}]\label{lemma:GHS}
		Suppose $X\sim \P_{\beta,\bQ,\bmu}$ with $\beta>0$, $\bQ_{ij}\geq 0$ for all $i,j\in [n]$ and 
	 $\bmu\in \left(\mathbb{R}^+\right)^{n}$. Then for any $(i_1,i_2,i_3)\in [n]^{\otimes 3}$ one has
	\begin{align*}
      \frac{\partial^3\log{Z_n(\beta,\bQ,\bmu)}}{\partial \mu_{i_1}\partial \mu_{i_2}\partial \mu_{i_3}}\leq 0.
	\end{align*}
	Consequently, for any $\bmu_1\succcurlyeq\bmu_2\succcurlyeq\mathbf{0}$ (i.e. coordinate-wise inequality) one has
\begin{align}
\mathrm{Cov}_{\beta,\bQ,\bmu_1}(X_i,X_j)&\leq \mathrm{Cov}_{\beta,\bQ,\bmu_2}(X_i,X_j),\label{eqn:correlation_ordering}
\end{align}
whenever $\beta\bQ_{ij}\geq 0$ for all $i,j\in [n]$.
\end{lemma}


\begin{lemma}[GKS Inequality \citep{friedli2017statistical}]\label{lemma:GKS}
		Suppose $X\sim \P_{\beta,\bQ,\bmu}$ with $\beta>0$, $\bQ_{ij}\geq 0$ for all $i,j\in [n]$ and $\bmu\in \left(\mathbb{R}^+\right)^{n}$. Then the following hold for any $i,j\in [n]$
		\begin{align*}
		\mathrm{Cov}_{\beta,\bQ,\bmu}(X_i,X_j)\geq 0; \quad \E_{\beta,\bQ,\bmu}(X_i)\geq 0.
		\end{align*}
\end{lemma}

\begin{lemma}[Lemma 8 of \cite{daskalakis2019testing}]\label{lemma:griffith_second}
	Suppose $X^{(k)}\sim \P_{\beta^{(k)},\bQ^{(k)},\mathbf{0}}$ for  $k=1,2$ with $\beta^{(1)}\bQ^{(1)}_{ij}\geq \beta^{(2)}\bQ^{(2)}_{ij}\geq 0$ for all $i,j$. 
	Then
	\begin{align*}
	\mathrm{Cov}_{\beta^{(1)},\bQ^{(1)},\mathbf{0}}(X_i,X_j)\geq \mathrm{Cov}_{\beta^{(2)},\bQ^{(2)},\mathbf{0}}(X_i,X_j), \quad \forall i,j.
	\end{align*}
\end{lemma}

\begin{lemma}[See Theorem 1.5 of~\cite{chatterjee2007stein}~and~Lemma 2.1(b) of~\cite{Deb2020}]\label{lemma:chatterjee}
Let $\bX\sim \PJ$. Then for any $t>0$ and $S\subset [n]$ we have
\begin{align*}
\PJ\left(|L_{S}(\bmu)|>2(1+\beta\|\bQ\|_{\infty \rightarrow \infty})t\right)&\leq 2e^{-t^2/2},
\end{align*} 
where
$L_{S}(\bmu):=\frac{1}{\sqrt{|S|}}\sum_{i\in S} (X_i-\tanh(\beta m_i+\mu_i))$ with $m_i=\sum_{j=1}^n \bQ_{ij}X_j.$
\end{lemma}

\begin{lemma}
\label{lemma:chatterjee_high_temp}

Let $\bX\sim \PJ$ such that $\beta \|\bQ\|_{\infty\rightarrow\infty}<1$ and $\beta\min_{i,j}\bQ_{i,j}\geq 0$. Then for any $S\subset [n]$  and $t>0$ we have

\begin{align*}
\P_{\bQ,\bmu}\Big(|\tilde{L}_{S}(\bmu)|>\frac{t}{\sqrt{1-\beta \|\bQ\|_{\infty\rightarrow\infty}}}\Big)\le 2 e^{-t^2},
\end{align*}
where $
\tilde{L}_{S}(\bmu):=\frac{1}{\sqrt{|S|}}\sum_{i \in S} (X_i-\E_{\bQ,\bmu}X_i)$.
\end{lemma}

\begin{lemma}\label{lemma:second_moment_monotonicity}
Consider a finite set $\bmu_1,\ldots,\bmu_k$ of $\left(\mathbb{R}^+\right)^{n}$ and  let $\pi$ be the uniform prior on them.  If $\beta\bQ_{ij}\geq 0$ for all $i,j\in [n]$ and $L_{\pi}$ denotes the likelihood ratio of \eqref{eqn:sparse_hypo} w.r.t $\pi$, then $\E_{\mathbf{\beta,\bQ,\mathbf{0}}}(L_{\pi}^2)$, viewed as a function of $k\times n$ coordinates of $\bmu_1,\ldots,\bmu_k$ is coordinate-wise increasing. 

\end{lemma}

\begin{lemma}\label{lemma:expectation_vs_externalmag}
	Suppose $X\sim \P_{\beta,\bQ,\bmu}$  with $\bmu\in (\mathbb{R}^+)^{n}$, and $\beta \bQ_{i,j}\ge 0$ for all $i\ne j$. Setting $\rho:=(1-\tanh(\beta \|\bQ\|_{\infty\rightarrow\infty})$ we have $\E_{\beta, \bQ,\bmu}(X_i)\ge \rho\tanh(\mu_i)$.
		\end{lemma}

\subsection{Proof of Theorem \ref{thm:upper}}

By  Lemma \ref{lemma:chatterjee} we get
\begin{align*}
\PZ\left(|L_{S}({\bf 0})|> 2(1+\beta \|\bQ\|_{\infty \rightarrow \infty})\sqrt{2(1+\delta)\log{|\mathcal{C}_n|}}\right)\leq 2\exp\left(-(1+\delta)\log{|\mathcal{C}_n|}\right).
\end{align*} 
A union bound then gives
\begin{align*}
\PZ\left(L_n>2(1+\beta\|\bQ\|_{\infty \rightarrow \infty})\sqrt{2(1+\delta)\log{|\mathcal{C}_n|}}\right)&\leq \frac{2}{|\mathcal{C}_n|^{\delta}}\stackrel{n \rightarrow \infty}{\rightarrow}0,
\end{align*}
yielding a control over the Type I error of the test. In order to control the Type II error of the test suppose $\mathrm{supp}(\bmu)=S$ for some $S\in \mathcal{C}_n$. Then we show that $|L_{S}(\mathbf{0})|$ beats the the null cut-off. To this end note that
\begin{align*}
|L_{S}(\mathbf{0})|\geq \left\vert\frac{1}{\sqrt{|S|}}\sum_{i\in S} (\tanh(\beta m_i+\mu_i)-\tanh(\beta m_i))\right\vert-|L_{S}(\bmu)|
\end{align*}
Again by Lemma \ref{lemma:chatterjee} we have
\begin{align*}
\P_{\beta,\bQ,\bmu}\left(|L_{S}(\bmu)|\leq 2(1+\beta \|\bQ\|_{\infty \rightarrow \infty})\sqrt{2(1+\delta)\log{|\mathcal{C}_n|}}\right)\geq 1-2\exp\left(-(1+\delta)\log{|\mathcal{C}_n|}\right).
\end{align*} 

On the other hand, observe that from elementary calculus,
	\begin{align*}
	\sup_{x\in [0,K],y\geq 0}\frac{\tanh(x+y)-\tanh(x)}{\tanh(y)}\gtrsim 1,
	\end{align*}
	where $K>0$ and the (hidden) constant on the right is strictly positive and depends on $K$. Therefore there exists a constant $M>0$ (depending on $\beta$, $\lVert \bQ\rVert_{\infty\to\infty}$) such that
\begin{align*}
\sum_{i\in S} (\tanh(\beta m_i+\mu_i)-\tanh(\beta m_i)
 \geq M\tanh(A)|S|.
\end{align*} 
The desired control on the Type II error therefore follows on noting the following string of inequalities:
 \begin{align*}
 M\tanh(A)\sqrt{|S|}\ge C'\cdot M\sqrt{\log n}
\ge 4(1+\beta \|\bQ\|_{\infty\rightarrow\infty})\sqrt{2(1+\delta)\log |\mathcal{C}_n|},
 \end{align*}
 where the last inequality holds for all $n$ large (using the fact that $\log{|\mathcal{C}_n|}\leq C_u\log{n}$) for $C':=\frac{8(1+\beta C_u')\sqrt{C_u}}{M}$.

\subsection{Proof of Theorem \ref{thm:upper_unknown_beta_Q}}


To begin, a union bound along with Lemma \ref{lemma:chatterjee_high_temp} gives
\begin{align*}
\PZ\left(\tilde{L}_n>\sqrt{\frac{(1+\delta)\log |\mathcal{C}_n|}{1-\eta}}\right)&\leq \frac{2}{|\mathcal{C}_n|^{\delta}}\stackrel{n \rightarrow \infty}{\rightarrow}0,
\end{align*}
yielding a control over the Type I error of the test. 

Proceeding to control Type II error, 
with $\rho:=1-\tanh(\eta)$ we have for the signal set $S\in \mathcal{C}_n$
\begin{align*}
\big|\tilde{L}_{S}(\mathbf{0})\big|\geq \frac{1}{\sqrt{|S|}}\sum_{i\in S} \E_{\beta,\bQ,\bmu}(X_i)-|\tilde{L}_{S}(\bmu)|\ge \rho\tanh(A)\sqrt{|S|}-|\tilde{L}_{S}(\bmu)|,
\end{align*}
where the last inequality uses Lemma \ref{lemma:expectation_vs_externalmag}. Also, again invoking Lemma \ref{lemma:chatterjee_high_temp} we have
$$\P_{\beta,\bQ,\bmu}\left(|\tilde{L}_S(\bmu)|\le \sqrt{\frac{(1+\delta) \log |\mathcal{C}_n|}{1-\eta}}\right)\ge 1-2\exp(-(1+\delta)\log |\mathcal{C}_n|).$$
This gives that
$$\tilde{L}_n\ge \tilde{L}_{S}({\bf 0})\ge \rho\tanh(A)\sqrt{|S|}-\sqrt{\frac{(1+\delta)\log |\mathcal{C}_n|}{1-\eta}}$$ with probability tending to $1$ under $\P_{\beta,\bQ,\bmu}$, and so Type II error converges to $0$ as soon as we have $$\rho\tanh(A)\sqrt{s}\ge 2\sqrt{\frac{(1+\delta)\log |\mathcal{C}_n|}{1-\eta}},$$
which can be achieved by choosing $C'=\frac{4\sqrt{C_u}}{\rho\sqrt{1-\eta}}$ (since $\log|\mathcal{C}_n|\leq C_u\log{n}$), which depends only on $\eta,C_u$.

\subsection{Proof of Theorem \ref{thm:lower_short_range_large_s}}

\begin{enumerate}
	
\item [(I)]

Let $\mathcal{C}_n'$ be the subclass of $\mathcal{C}_n$ described in the statement of Theorem \ref{thm:lower_short_range_large_s}. Recall that for any $S\in \mathcal{C}_n$ and real number $\eta$, $\bmu_S(\eta)$ denotes the vector which has $\mu_i=\eta\I(i\in S)$. 
Letting $\pi$ denote the uniform prior on $\{\P_{\beta,\bQ,\bmu_S(A)}, S\in \mathcal{C}_n'\}$ it is easy to see that the corresponding second moment of the likelihood ratio is given by (owing to the disjointedness of the sets in $\mathcal{C}_n'$),
\begin{align}
\E_{\beta,\bQ,\mathbf{0}}(L_{\pi}^2)&=\frac{1}{|\mathcal{C}_n'|^2}\sum_{S\in \mathcal{C}_n'}\frac{Z_n^2(\beta,\bQ,\mathbf{0})Z_n(\beta,\bQ,\bmu_S(2A))}{Z_n^2(\beta,\bQ,\bmu_S(A))Z_n(\beta,\bQ,\mathbf{0})}\nonumber\\
&+\frac{1}{|\mathcal{C}_n'|^2}\sum_{S_1\neq S_2\in \mathcal{C}_n'}\frac{Z_n^2(\beta,\bQ,\mathbf{0})Z_n(\beta,\bQ,\mathbf{\bmu}_{S_1\cup S_2}(A))}{Z_n(\beta,\bQ,\bmu_{S_1}(A))Z_n(\beta,\bQ,\bmu_{S_2}(A))Z_n(\beta,\bQ,\mathbf{0})}.\label{eqn:second_monent_short_range}
\end{align}
Now for any $S\in \mathcal{C}_n'$ a two term Taylor expansion in $A$ around $0$  gives the existence of $\eta\in [0,A]$ (depending on $S,A,\beta$) 
such that
\begin{align*}
\ & \frac{Z_n^2(\beta,\bQ,\mathbf{0})Z_n(\beta,\bQ,\bmu_S(2A))}{Z_n^2(\beta,\bQ,\bmu_S(A))Z_n(\beta,\bQ,\mathbf{0})}\\
&=\exp\left(\log{Z_n(\beta,\bQ,\mathbf{0})}+\log{Z_n(\beta,\bQ,\bmu_S(2A))}-2\log{Z_n(\beta,\bQ,\bmu_S(A))}\right)\\
&=\exp\left(\frac{A^2}{2}\left[4\frac{\partial^2 \log Z_n(\beta,\bQ,\bmu_S(2h))}{\partial h^2}\vert_{h=\eta}-2\frac{\partial^2\log Z_n(\beta,\bQ,\bmu_S(h))}{\partial h^2}\vert_{h=\eta}\right]\right)\\
&=\exp\left(\frac{A^2}{2}\left[4\mathrm{Var}_{\beta,\bQ,\bmu_S(2\eta)}\left(\sum_{i\in S}X_i\right)-2\mathrm{Var}_{\beta,\bQ,\bmu_S(\eta)}\left(\sum_{i\in S}X_i\right)\right]\right)\\
&\leq \exp\left(\frac{A^2}{2}\left[4\mathrm{Var}_{\beta,\bQ,\bmu_S(2\eta)}\left(\sum_{i\in S}X_i\right)+2\mathrm{Var}_{\beta,\bQ,\bmu_S(\eta)}\left(\sum_{i\in S}X_i\right)\right]\right)\\
&\leq \exp\left(\frac{2A^2}{2}\left[4\sum_{i,j\in S}\mathrm{Cov}_{\beta,\bQ,\bmu_S(2\eta)}(X_i,X_j)+2\sum_{i,j\in S}\mathrm{Cov}_{\beta,\bQ,\bmu_S(\eta)}(X_i,X_j)\right]\right).
\end{align*}

Now the main challenge is to understand these spin-spin covariances at arbitrary magnetization $\eta$ at locations $S$. To deal with this we employ GHS inequality (Lemma \ref{lemma:GHS}) to get that for any $\eta\geq 0$, any $S\in \mathcal{C}_n'$, and any $i,j$ we have
\begin{align*}
\mathrm{Cov}_{\beta,\bQ,\bmu_S(\eta)}\left(X_i,X_j\right)\leq \mathrm{Cov}_{\beta,\bQ,\mathbf{0}}\left(X_i,X_j\right).
\end{align*}

Therefore, using the condition of the theorem we have that

\begin{align}
\frac{Z_n^2(\beta,\bQ,\mathbf{0})Z_n(\beta,\bQ,\bmu_S(2A))}{Z_n^2(\beta,\bQ,\bmu_S(A))Z_n(\beta,\bQ,\mathbf{0})}&\leq
\exp\left(6A^2\mathrm{Var}_{\beta,\bQ,\mathbf{0}}\sum_{i\in S} X_i\right)\leq  \exp\left(6A^2r_n\right).\label{eqn:second_moment_shortrange_term1}
\end{align}

Next note that, once again for any $S_1,S_2\in \mathcal{C}_n'$ we have for some $\eta\in [0,A]$ (possibly different) such that the following hold by a two term Taylor expansion in $A$ around $0$: 
\begin{align*}
\ & \frac{Z_n^2(\beta,\bQ,\mathbf{0})Z_n(\beta,\bQ,\mathbf{\bmu}_{S_1\cup S_2}(\eta))}{Z_n(\beta,\bQ,\bmu_{S_1}(A))Z_n(\beta,\bQ,\bmu_{S_2}(A))Z_n(\beta,\bQ,\mathbf{0})}\\
&=\exp\left(\frac{A^2}{2}\left[\begin{array}{c}\mathrm{Var}_{\beta,\bQ,\mathbf{\bmu}_{S_1\cup S_2}(\eta)}\left(\sum_{i\in S_1\cup S_2}X_i\right)\\-\mathrm{Var}_{\beta,\bQ,\bmu_{S_1}(\eta)}\left(\sum_{i\in S_1}X_i\right)-\mathrm{Var}_{\beta,\bQ,\bmu_{S_2}(\eta)}\left(\sum_{i\in S_2}X_i\right)\end{array}\right]\right).
\end{align*}
Again by GHS inequality (Lemma \ref{lemma:GHS}) one has that for disjoint $S_1,S_2$ and $\eta\geq 0$
\begin{align*}
\mathrm{Var}_{\beta,\bQ,\mathbf{\bmu}_{S_1\cup S_2}(\eta)}\left(\sum_{i\in S_1}X_i\right)\le \mathrm{Var}_{\beta,\bQ,\bmu_{S_1}(\eta)}\left(\sum_{i\in S_1}X_i\right),\\
\mathrm{Var}_{\beta,\bQ,\mathbf{\bmu}_{S_1\cup S_2}(\eta)}\left(\sum_{i\in S_2}X_i\right)\le \mathrm{Var}_{\beta,\bQ,\bmu_{ S_2}(\eta)}\left(\sum_{i\in  S_2}X_i\right).
\end{align*}
Consequently, 
\begin{align}
\ & \frac{Z_n^2(\beta,\bQ,\mathbf{0})Z_n(\beta,\bQ,\bmu_{S_1\cup S_2}(A))}{Z_n(\beta,\bQ,\bmu_{S_1}(A))Z_n(\beta,\bQ,\bmu_{S_2}(A))Z_n(\beta,\bQ,\mathbf{0})}\nonumber\\
&\leq \exp\left(\frac{A^2}{2}\sum_{i\in S_1, j\in S_2}\mathrm{Cov}_{\beta,\bQ,\mathbf{\bmu}_{S_1\cup S_2}(\eta)}(X_i,X_j)\right)\nonumber\\
&\leq \exp\left(\frac{A^2}{2}\sum_{i\in S_1, j\in S_2}\mathrm{Cov}_{\beta,\bQ,\mathbf{0}}(X_i,X_j)\right)
\end{align}
where the second to last line follows, as before, by GHS inequality.
Therefore, combining \eqref{eqn:second_monent_short_range} and \eqref{eqn:second_moment_shortrange_term1}, we have 
\begin{align}
\E_{\beta,\bQ,\mathbf{0}}(L_{\pi}^2)&\leq \frac{1}{|\mathcal{C}_n'|^2}\sum_{S\in \mathcal{C}_n'}\exp\left(6A^2r_n \right) +\frac{1}{|\mathcal{C}_n'|^2}\sum_{S_1\neq S_2\in \mathcal{C}_n'}\exp\left(o(A^2r_n')\right).\label{eqn:second_moment_short_range_generalclass}
\end{align}

As $\log{|\mathcal{C}_n'|}\geq C_l\log{n}$, there exists constant $c'>0$ such that if $A\leq c'\min\left\{\sqrt{\frac{\log{n}}{r_n}},\frac{1}{\sqrt{r_n'}}\right\}$ then one has $\E_{\beta,\bQ,\mathbf{0}}(L_{\pi}^2)=1+o(1)$ by \eqref{eqn:second_moment_short_range_generalclass}. Since $r_n\geq C\log{n}$, the same conclusion holds if $\tanh(A)\leq c'\min\left\{\sqrt{\frac{\log{n}}{r_n}},\frac{1}{\sqrt{r_n'}}\right\}$ for a different constant $c'>0$.

\item [(II)] 
To prove this part of theorem, we consider the same prior $\pi$ as in part (I) and denote $P_{\beta,\bQ,\pi}$ to be the corresponding mixture of probability  measures. We first claim that it is enough to prove that $\mathbb{P}_{\beta,\bQ,\pi}(\cdot|\Omega_n)$ is contiguous w.r.t. $\mathbb{P}_{\beta,\bQ,\mathbf{0}}(\cdot|\Omega_n)$ (where for any distribution $\P$ of $ \bX$, $\P(\cdot|\Omega)$ is used denote conditional distribution given $\bX\in \Omega$). To verify this by contradiction, suppose we have a sequence of rejection regions $\mathcal{R}_n$ such that $$
\P_{\beta,\bQ,\mathbf{0}}(\mathcal{R}_n)\rightarrow 0,\quad 
\P_{\beta,\bQ,\pi}(\mathcal{R}_n)\rightarrow 1.$$ We will show that if $\mathbb{P}_{\beta,\bQ,\pi}(\cdot|\Omega_n)$ is contiguous w.r.t. $\mathbb{P}_{\beta,\bQ,\mathbf{0}}(\cdot|\Omega_n)$ then $\limsup_{n\rightarrow \infty}\P_{\beta,\bQ,\pi}(\mathcal{R}_n)\leq 1-\kappa$, which will give a contradiction. To see this, first note that since $\P_{\beta,\bQ,\mathbf{0}}(\mathcal{R}_n)\rightarrow 0$ and $\liminf_{n\rightarrow \infty}\P_{\beta,\bQ,\mathbf{0}}(\Omega_n)\geq \kappa>0$, we have $\P_{\beta,\bQ,\mathbf{0}}(\mathcal{R}_n|\Omega_n)\rightarrow 0$. Consequently, by contiguity one must have $\P_{\beta,\bQ,\pi}(\mathcal{R}_n|\Omega_n)\rightarrow 0$. Also since $\Omega_n$ is an increasing event for every $n$, we also have by monotonicity of measures  w.r.t. $\bmu$ 
that $\liminf_{n\rightarrow \infty}\P_{\beta,\bQ,\pi}(\Omega_n)\geq \kappa$.  Thus writing
\begin{align*}
\P_{\beta,\bQ,\pi}(\mathcal{R}_n)&= \P_{\beta,\bQ,\pi}(\mathcal{R}_n|\Omega_n)\P_{\beta,\bQ,\pi}(\Omega_n)+\P_{\beta,\bQ,\pi}(\mathcal{R}_n\cap \Omega_n^c),
\end{align*}
the first term of the right hand side of the display above goes to $0$ by contiguity and the second term satisfies $$\limsup_{n\rightarrow \infty}\P_{\beta,\bQ,\pi}(\mathcal{R}_n\cap \Omega_n^c)\leq \limsup_{n\rightarrow \infty}\P_{\beta,\bQ,\pi}(\Omega_n^c)\leq 1-\kappa.$$
It follows that $\limsup\limits_{n\to\infty}\P_{\beta,\bQ,\pi}(\mathcal{R}_n)\le 1-\kappa$, as desired.
\\

It thus suffices to verify conditional contiguity, which follows 
$\E_{\beta,\bQ,\mathbf{0}}(L^2_{\pi,\Omega_n}|\Omega_n)$ stays bounded, where 
$L_{\pi,\Omega_n}$ is likelihood ratio of the probability measure $\mathbb{P}_{\beta,\bQ,\pi}(\cdot|\Omega_n)$ is contiguous w.r.t. $\mathbb{P}_{\beta,\bQ,\mathbf{0}}(\cdot|\Omega_n)$. To this end, one has that the conditional probability measure $\P_{\beta,\bQ,\bmu}(\cdot|\Omega_n)$ corresponding to any $\P_{\beta,\bQ,\bmu}$ in \eqref{eqn:general_ising} is given by
\begin{align*}
\P_{\beta,\bQ,\bmu}(\bX=\bx|\Omega)=\frac{\exp\left(\frac{\beta}{2}\bx^T\bQ\bx+\bmu^T\bX\right)\mathbf{1}(x\in \Omega_n)}{Z_n(\beta,\bQ,\bmu|\Omega_n)},
\end{align*}
where $Z_n(\beta,\bQ,\bmu|\Omega_n)=\sum_{\bx\in \Omega_n}\exp\left(\frac{\beta}{2}\bx^T\bQ\bx+\bmu^T\bX\right)$. Therefore, by direct calculations similar to part (I) one has 
\begin{align}
\E_{\beta,\bQ,\mathbf{0}}(L^2_{\pi,\Omega_n}|\Omega_n)&=\frac{1}{|\mathcal{C}_n'|^2}\sum_{S\in \mathcal{C}_n'}\frac{Z_n^2(\beta,\bQ,\mathbf{0}|\Omega_n)Z_n(\beta,\bQ,\bmu_S(2A)|\Omega_n)}{Z_n^2(\beta,\bQ,\bmu_S(A)|\Omega_n)Z_n(\beta,\bQ,\mathbf{0}|\Omega_n)}\nonumber\\
&+\frac{1}{|\mathcal{C}_n'|^2}\sum_{S_1\neq S_2\in \mathcal{C}_n'}\frac{Z_n^2(\beta,\bQ,\mathbf{0}|\Omega_n)Z_n(\beta,\bQ,\mathbf{\bmu}_{S_1\cup S_2}(\eta)|\Omega_n)}{Z_n(\beta,\bQ,\bmu_{S_1}(A)|\Omega_n)Z_n(\beta,\bQ,\bmu_{S_2}(A))Z_n(\beta,\bQ,\mathbf{0}|\Omega_n)}\\
&=I+II\label{eqn:conditional_second_moment}
\end{align}
Now, it is easy to check by direct calculations that for any $S\in\mathcal{C}_n'$,
\begin{align*}
\frac{\partial \log Z_n(\beta,\bQ,\bmu_S(h)|\Omega_n)}{\partial h}\vert_{h=\eta}&=\E_{\beta,\bQ,\bmu_S(\eta)}\left(\sum_{i\in S}X_i|\Omega_n\right),\\
\frac{\partial^2 \log Z_n(\beta,\bQ,\bmu_S(h)|\Omega_n)}{\partial h^2}\vert_{h=\eta}&=\mathrm{Var}_{\beta,\bQ,\bmu_S(\eta)}\left(\sum_{i\in S}X_i|\Omega_n\right)
\end{align*}
This implies that, the first term of \eqref{eqn:conditional_second_moment} can be bounded similar to part (I) as
\begin{align}
I\leq \frac{1}{|\mathcal{C}_n'|^2}\sum_{S\in \mathcal{C'}}\exp\left(4\frac{A^2}{2}\mathrm{Var}_{\beta,\bQ,\mu_S(2\eta)}\left(\sum_{i\in S}X_i|\Omega_n\right)\right)\label{eqn:conditional_second_moment_I}
\end{align}
for some $0\leq \eta\leq A$ 
Similarly, the second term of \eqref{eqn:conditional_second_moment} can be written as
\begin{align}
II=\exp\left(\frac{A^2}{2}\left[\begin{array}{c}\mathrm{Var}_{\beta,\bQ,\mathbf{\bmu}_{S_1\cup S_2}(\eta)}\left(\sum_{i\in S_1\cup S_2}X_i|\Omega_n\right)\\-\mathrm{Var}_{\beta,\bQ,\bmu_{S_1}(\eta)}\left(\sum_{i\in S_1}X_i|\Omega_n\right)-\mathrm{Var}_{\beta,\bQ,\bmu_{S_2}(\eta)}\left(\sum_{i\in S_2}X_i|\Omega_n\right)\end{array}\right]\right).\label{eqn:conditional_second_moment_II}
\end{align}
The conclusion then follows using the given assumptions, in a similar manner as in part (I).

\end{enumerate}

\subsection{Proof of Theorem \ref{thm:lower_short_range_small_s}}
\begin{enumerate}
\item[(a)]
With $\pi(A)$ denoting the same prior as in the proof of Theorem \ref{thm:lower_short_range_large_s}, the likelihood ratio $L_{\pi(A)}$ is given by 
\begin{align*}
L_{\pi(A)}({\bf x})=\frac{1}{|\mathcal{C}_n'|}\sum_{S\in \mathcal{C}_n'}\frac{\P_{\bbeta,\bQ,\bmu_S(A)}(\bX={\bf x})}{\P_{\bbeta,\bQ,\mathbf{0}}(\bX={\bf x})}
=\frac{1}{|\mathcal{C}_n'|}\sum_{S\in \mathcal{C}_n'}\frac{\P_{\bbeta,\bQ,\bmu_S(A)}(\bX_S={\bf x}_S)}{\P_{\bbeta,\bQ,\mathbf{0}}(\bX_S={\bf x}_S)},
\end{align*}
where the second inequality follows on noting that the conditional distribution of $\bX_{S^c}$ given $\bX_S$ is the same under both the measures $\P_{\beta,\bQ,\bmu_S(A)}$ and $\P_{\beta,\bQ,\mathbf{0}}$. On letting $A\rightarrow\infty$ gives
\begin{align*}
\lim_{A\rightarrow\infty}L_\pi({\bf x})=\frac{1}{|\mathcal{C}_n'|}\sum_{S\in \mathcal{C}_n'}\frac{\mathbf{1}({\bf x}_s={\bf 1})}{\P_{\bbeta,\bQ,\mathbf{0}}(\bX_S={\bf x}_S)}=:L_\infty, \text{ say}.
\end{align*}
Since, by Lemma \ref{lemma:second_moment_monotonicity}, $\E_{\beta,\bQ,\mathbf{0}} L_{\pi(A)}^2$ is a non-decreasing function of $A$, and $L_{\pi(A)}\le  \frac{1}{\P_{\beta,Q,0}(\bX_S={\bf x}_S)}$ which is bounded in $A$, using Dominated Convergence we have
\begin{align}\label{eq:s_less}
\notag \E_{\beta,\bQ,\mathbf{0}} L_{\pi(A)}^2&\le \lim_{A\rightarrow\infty}\E_{\beta,\bQ,\mathbf{0}} L_{\pi(A)}^2\\
&=\E_{\beta,\bQ,\mathbf{0}} L_\infty^2\\
\notag&=\frac{1}{|\mathcal{C}_n'|^2}\sum_{S\in \mathcal{C}_n'}\frac{1}{\P_{\bbeta,\bQ,\mathbf{0}}(\bX_S={\bf 1})}+\frac{1}{|\mathcal{C}_n'|^2}\sum_{S_1,S_2\in \mathcal{C}_n'}\frac{\P_{\bbeta,\bQ,\mathbf{0}}(\bX_{S_1\cup S_2}={\bf 1})}{\P_{\bbeta,\bQ,\mathbf{0}}(\bX_{S_1}={\bf 1}) \P_{\bbeta,\bQ,\mathbf{0}}(\bX_{S_2}={\bf 1})}\\
& \le \frac{2^s}{|\mathcal{C}_n'|}+\sup_{S_1\ne S_2}\frac{\P_{\bbeta,\bQ,\mathbf{0}}(\bX_{S_1\cup S_2}={\bf 1})}{\P_{\bbeta,\bQ,\mathbf{0}}(\bX_{S_1}={\bf 1}) \P_{\bbeta,\bQ,\mathbf{0}}(\bX_{S_2}={\bf 1})}.
\end{align}
The first term in the RHS of \eqref{eq:s_less} is small since $\log{|\mathcal{C}_n'|}\geq C_l\log{n}$ and $s\leq c\log n$ for a small enough $c>0$. The second term converges to $1$ using the given hypothesis. Thus we have $ \E L_\infty^2=1+o(1)$, and so the proof is complete.

\item[(b)]
It suffices to show that $\E_{\beta,\bQ,\mathbf{0}} L_{\pi(A)}^2=O(1)$. We can assume that $\PZ (\Omega_n)\geq \kappa/2$ for all large $n$ and some $\kappa>0$. Using \eqref{eq:s_less} it suffices to show that
\begin{align}\label{eq:combine}
\sup_{S_1\ne S_2\in \mathcal{C}_n'}\frac{\P_{\beta,\bQ,\mathbf{0}}(\bX_{S_1}={\bf 1},\bX_{S_2}={\bf 1})}{\P_{\beta,\bQ,\mathbf{0}}(\bX_{S_1}={\bf 1})\P_{\beta,\bQ,\mathbf{0}}(\bX_{S_2}={\bf 1})}\le \frac{4}{\kappa^2}+o(1).
\end{align}
To this effect, for $i=1,2$ a simple inclusion gives
\[\P_{\beta,\bQ,\mathbf{0}}(\bX_{S_i}={\bf 1})\ge \P_{\beta,\bQ,\mathbf{0}}(\bX_{S_i}={\bf 1},\Omega_n)\ge \frac{\kappa}{2}\P_{\beta,\bQ,\mathbf{0}}(\bX_{S_i}={\bf 1}|\Omega_n),\]
and the FKG inequality subsequently gives
\[\P_{\beta,\bQ,\mathbf{0}}(\bX_{S_1}={\bf 1},\bX_{S_2}={\bf 1})\le \P_{\beta,\bQ,\mathbf{0}}(\bX_{S_1}={\bf 1},\bX_{S_2}={\bf 1}|\Omega_n).\]
Combining these two observations along with the given hypothesis, \eqref{eq:combine} follows.


\end{enumerate}

\subsection{Proofs of Lemmas from~\ref  {section:supporting_lemmas}}\label{sec:pfsecsupplem}

The proofs of Lemmas~\ref{lemma:GHS}--\ref{lemma:chatterjee} follow from the references cited in the statements themselves.

\subsubsection{Proof of Lemma~\ref{lemma:chatterjee_high_temp}}

A direct computation gives
\begin{align*}
d_{TV}\left(\P^{(i)}(\cdot|\bx_{-i}),\P^{(i)}(\cdot|\by_{-i})\right)&=\frac{1}{2}\Big|\tanh\left(\beta \sum_{j}\bQ_{ij}x_j+\mu_i\right)-\tanh\left(\beta\sum_{j}\bQ_{ij}y_j+\mu_i\right)\Big|\\
&\le \beta \sum_{j}\bQ_{ij}1\{x_j\ne y_j\}.
\end{align*}
This, along with assumption $0\leq \beta \|\bQ\|_2<1$, on invoking \citep{chatterjee2005concentration}[Theorem 4.3] gives
the desired conclusion.

\subsubsection{Proof of Lemma~\ref{lemma:second_moment_monotonicity}}

	Note that
	\begin{align*}
	k^2\E_{\mathbf{\beta,\bQ,\mathbf{0}}}(L_{\pi}^2)&=\sum_{l}\frac{Z_n(\beta,\bQ,\mathbf{0})Z_n(\beta,\bQ,2\bmu_{l})}{Z_n^2(\beta,\bQ,\bmu_{l})}+\sum_{l_1\neq l_2}\frac{Z_n(\beta,\bQ,\mathbf{0})Z_n(\beta,\bQ,\bmu_{l_1}+\bmu_{l_2})}{Z_n(\beta,\bQ,\bmu_{l_1})Z_n(\beta,\bQ,\bmu_{l_2})}\\
	&=\sum_{l}\exp\left(\log Z_n(\beta,\bQ,\mathbf{0}))+\log Z_n(\beta,\bQ,2\bmu_{l})-2\log Z_n(\beta,\bQ,\bmu_{l})\right)\\
	+&\sum_{l_1\neq l_2}\exp\left(\log Z_n(\beta,\bQ,\mathbf{0}))+\log{Z_n(\beta,\bQ,\bmu_{l_1}+\bmu_{l_2})}-\log{Z_n(\beta,\bQ,\bmu_{l_1})-\log{Z_n(\beta,\bQ,\bmu_{l_2})}}\right)
	\end{align*} fix any coordinate $l=1,\ldots,k$ 
 and consider  $k^2\E_{\mathbf{\beta,\bQ,\mathbf{0}}}(L_{\pi}^2)$ as a function of the $n$ coordinates of $\bmu_l$ fixing the rest coordinates. We note that it is enough to show that each coordinate of this gradient is non-negative in the direction of any vector in $\left(\mathbb{R}^+\right)^{n}$. Being a sum of exponentials, it is sufficient to individually consider each exponent and show the same conclusion desired above. A typical such term is one of two types:
 \begin{align}
 2\frac{\partial \log{Z_n(\beta,\bQ,\bmu)}}{\partial \bmu}\vert_{\bmu=2\bmu_l}-2\frac{\partial \log{Z_n(\beta,\bQ,\bmu)}}{\partial \bmu}\vert_{\bmu=\bmu_l} \label{eqn:derivative_1}\quad \text{or} \\
  \frac{\partial \log{Z_n(\beta,\bQ,\bmu+\bmu_{l'})}}{\partial \bmu}\vert_{\bmu=\bmu_l}-\frac{\partial \log{Z_n(\beta,\bQ,\bmu)}}{\partial \bmu}\vert_{\bmu=\bmu_l}, l\neq l'. \label{eqn:derivative_2}
 \end{align}
 However, by mean value theorem, the $i^{\mathrm{th}}$ coordinate of $2\frac{\partial \log{Z_n(\beta,\bQ,\bmu)}}{\partial \bmu}_{\bmu=2\bmu_l}-2\frac{\partial \log{Z_n(\beta,\bQ,\bmu)}}{\partial \bmu}_{\bmu=\bmu_l}$ equals $\bmu_l^T\mathrm{Var}_{\beta,\bQ,\pmb{\eta}}(\bX)\mathbf{e}_i$ for some $\pmb{\eta}$ lying on the line joining $\bmu_l$ and $2\bmu_l$ and $\mathbf{e}_i$ denoting the $i^{\mathrm{th}}$ unit vector in $\mathbb{R}^{n}$. This implies that all coordinates of $\pmb{\eta}$ are positive. Consequently, each coordinate of $\mathrm{Var}_{\beta,\bQ,\pmb{\eta}}(\bX)$ is positive, since $\beta\bQ_{j_1j_2}\geq 0$ for all $j_1,j_2\in [n]$ which gives $\mathrm{Cov}_{\beta,\bQ,\pmb{\eta}}(X_{j_1}X_{j_2})\geq 0$ by GKS inequality (Lemma \ref{lemma:GKS}). This proves that the first term \eqref{eqn:derivative_1} has positive coordinates. A similar proof works for the second term \eqref{eqn:derivative_2}.
 
\subsubsection{Proof of Lemma~\ref{lemma:expectation_vs_externalmag}}
	Note that
	\begin{align*}
	\E_{\beta,\bQ,\bmu}(X_i)=&\E_{\beta,\bQ,\bmu}\tanh(\beta \sum_{j\in [n]}\bQ_{ij}X_j+\mu_i)\\\ge &\E_{\beta,\bQ,\bmu}\tanh(\beta \sum_{j\in [n]}\bQ_{ij}X_j)+(1-\tanh(\beta \|\bQ\|_{\infty\rightarrow\infty})\tanh(\mu_i),
	\end{align*}
	from which the result follows on noting that 
	$$\E_{\beta,\bQ,\bmu}\tanh(\beta \sum_{j\in [n]}\bQ_{ij}X_j)\ge \E_{\beta,\bQ,{\bf 0}}\tanh(\beta \sum_{j\in [n]}\bQ_{ij}X_j)=0.$$
	In the above display, the first inequality follows the fact that the Ising model is stochastically non decreasing in $\bmu$, along with the observation that the function $(x_j,j \in [n])\mapsto \tanh(\beta \sum_{j\in [n]}\bQ_{ij}x_j)$ is non decreasing, and the second equality follows by symmetry of the Ising model when $\bmu={\bf 0}$.

\section{Proofs of Theorems~\ref{thm:dense_regular}--\ref{thm:lattice}}\label{sec:corpf}
This section will be devoted to proving Theorems~\ref{thm:dense_regular}--\ref{thm:lattice}. Towards that direction, we first mention a collection of lemmas, whose proofs we defer. 

\subsection{Some Auxiliary Lemmas}
The first lemma describes some relevant properties of a fixed-point equation which arises naturally in Mean-Field Ising models (see~\cite{basak2017universality} for details) and will be useful for the subsequent discussion.

\begin{lemma}[See Page 10 in~\cite{AmirAndrea2010}]\label{lem:fixsol}
	Consider the fixed point equation 
	\begin{equation}\label{eq:mainfix}
\phi(x)=0,\text{ where } \phi(x):=x-\tanh(\beta x+B).
\end{equation}
	\begin{enumerate}
		\item[(a)] (High temperature) If $\beta<1$, then~\eqref{eq:mainfix} has a unique solution at $t=0$, and $\phi'(0)>0$.
		\item[(b)] (Low temperature) If $\beta>1$, then \eqref{eq:mainfix} has two non zero roots $\pm t$ of this equation, where $t>0$, and $\phi'(\pm t)>0$.
		\item[(c)] (Critical temperature) If $\beta=1$,  then~\eqref{eq:mainfix} has a unique solution at $t=0$, and $\phi'(0)=0$.
	\end{enumerate} 
\end{lemma} 
In the rest of the paper, $t$ will always denote the nonnegative root of $\phi(\cdot)$ as defined in~\ref  {lem:fixsol}.

For the remaining results we need a few notation. For a graph $\mathbb{G}_n=(\mathcal{V}_n,\mathcal{E}_n)$ with vertex set $\mathcal{V}_n$ and edge-set $\mathcal{E}_n$, let $\mathbf{G}_n$ denote the adjacency matrix with its $(i,j)^{\mathrm{th}}$ element denoted by $\mathbf{G}_n(i,j)$. Let $d_i$ denote the degree of vertex $i\in \mathcal{V}_n$, $\overline{d}$ the average degree, and $d_{\max}$ the maximum degree. For an Ising model $\P_{\beta,\bQ,\bmu}$ defined on $\mathbb{G}_n$ we will use the convention that $\bQ=\mathbf{G}/\overline{d}$. Finally, we denote the $i^{\mathrm{th}}$ largest eigenvalue of a square matrix $\mathbf{M}$ by $\lambda_i(\mathbf{M})$. With these notation, the following lemma establishes sharp bounds on the spin-spin correlations for some Mean-Field type Ising models. These will serve as quintessential ingredients for verifying the conditions of Theorem \ref{thm:lower_short_range_large_s} for the examples in Section \ref{section:Examples}.

\begin{lemma}\label{lem:slarge}
	Let $\alpha_n:=\sqrt{\frac{\log{n}}{\overline{d}}}$ and assume that $\max_{i\in [n]}\Big|\frac{d_i}{\overline{d}}-1\Big|\to 0$. 
	\begin{enumerate}
	\item[(a)] If $0\leq\beta<1$ and $\overline{d}\gtrsim (\log n)^2$ then we have:
	$$|\EZ (X_iX_j)|\lesssim \begin{cases}\frac{1}{\overline{d}}& \mbox{if }(i,j)\in \mathcal{E}_n\\ \bigg(\max_{(i,j)} (\bQ^3)_{ij}\bigg)+\alpha_n^4 &\mbox{if }(i,j)\notin \mathcal{E}_n \end{cases},$$
	where $\mathcal{E}_n$ denotes the set of edges in $\mathbb{G}_n$.
	
	\item[(b)] If $\beta>1$ and $\limsup\limits_{n\to\infty} \max_{i\in [n]}\beta \sech^2(\beta t+\mu_i)< 1$ and the assumptions in Lemma \ref{lem:mcontrol} [part (b)(i) or (b)(ii)] hold, then we have:
	
	(i)\begin{align*}
	\big|\CG (X_i,X_j|\bar{\bX}\geq 0) &\lesssim \begin{cases} \frac{1}{\overline{d}}& \mbox{if }(i,j)\in \mathcal{E}_n\\ \bigg(\max_{(i,j)}(\bQ^3)_{ij}\bigg)+\alpha_n^3 &\mbox{if }(i,j)\notin \mathcal{E}_n \end{cases}.
	\end{align*}
(ii) $$\max_{i\in [n]}|\VG (X_i|\bar{\bX}\geq 0)-\sech^2(\beta t+\mu_i)|\lesssim \alpha_n.$$

	\item[(c)] If $\beta=1$, \eqref{eq:connected} holds, and the graph $\mathbb{G}_n$ satisfies  
	\[\bd\gtrsim \sqrt{n}(\log{n})^5,\quad \max_{i\in [n]}\Big|\frac{d_i}{\overline{d}}-1\Big|\lesssim \alpha_n,\] then we have
	$|\EZ [X_iX_j]|\lesssim n^{-1/2}.$
		\end{enumerate}
\end{lemma}

Our next result establishes some crucial probability estimates for Mean-Field type Ising models. These will serve as quintessential ingredients for verifying the conditions of Theorem \ref{thm:lower_short_range_small_s} for the examples in Section \ref{section:Examples}.

\begin{lem}\label{lem:smallinsig} 
	Assume that $\max_{i\in [n]}|d_i/\bd-1|\to 0$.
\begin{enumerate}
	\item[(a)] If $0\leq \beta< 1$, then for any fixed $c>0$ and any $s$ satisfying $s\le c\log n$, we have
	\begin{equation}\label{eq:smallinsig}\lim_{n\rightarrow\infty}\sup_{S:|S|=s,\ {\bf a}\in \{-1,1\}^s}\Bigg|\frac{\PZ (\bX_S={\bf a})}{2^{-s}}-1\Bigg|=0,\end{equation}
	provided $\bd\gg (\log{n})^4$.
	
	\item[(b)]
	If $\beta=1$, same conclusion as part (a) holds provided $\bd \gg (\log n)^{10}$.
	
	\item[(c)] 
 If $\beta>1$ and the assumptions in either part (b)(i) or part (b)(ii) of Lemma~\ref{lem:mcontrol} hold, then we get:
	\begin{equation}\label{eq:smallinsignouniq}\lim_{n\rightarrow\infty}\sup_{S:|S|=s,\ {\bf a}\in \{-1,1\}^s}\Bigg|\frac{\PZ (\bX_S={\bf a}|\bar{\bX}\geq 0)}{g({\bf a},s)}-1\Bigg|=0,\quad g({\bf a},s):=\frac{\exp(\beta t\sum_{i=1}^{s} a_i)}{\sum_{{\bf b}\in \{-1,1\}^{s}} \exp(\beta t\sum_{i=1}^{s} b_i)},
	\end{equation}
	provided $\overline{d}\gg (\log n)^4$.
	\end{enumerate}
\end{lem}

Our final result in this section establishes precise behavior of average magnetization $\bar{X}$ at critical temperature for some Mean-Field type Ising models -- under the presence of asymptotically vanishing, yet detectable, external magnetization $\bmu$. The application of this result for $s\gtrsim \sqrt{n}/\log{n}$ in these models yields matching sharp upper bounds to the lower bounds developed in Theorem \ref{thm:lower_short_range_large_s}.

\begin{lemma}\label{lem:altbeh}
	Suppose $\bmu\in {\Xi}(\mathcal{C}_n,s,A)$, $\beta=1$, $\bd \gg \sqrt{n}(\log{n})^5$ and $sA \gg n^{1/4}$. Further assume that $\max_{i\in [n]}|d_i/\bd-1|\lesssim \sqrt{\frac{\log{n}}{\bd}}$,  and $\limsup_{n\to\infty}\lambda_2(\bQ)<1$. Then there exists a constant $\delta>0$ such that 
	\begin{align*}
	\lim_{n\to\infty} \PG(n^{1/4}\overline{\bX}\geq \delta k_n)\to 1
	\end{align*}
	where $k_n:=(n^{-1/4}sA)^{1/3}$.
\end{lemma}

Our final lemma concerns the correlation decay property in the context of Ising models on lattices  which serves as the main tool in the proof of Theorem \ref{thm:lattice}. We refer to \cite{aizenman1987phase,duminil2016new,duminil2017lectures,duminil2017sharp,mukherjee2019testing} for more details. We use the notation used in Section \ref{section:lattice} for denoting the vertices of the $d$-dimensional lattice.

\begin{lemma}[Correlation Decay ]
\label{lemma:correlation_decay}
	Suppose $X\sim \P_{\beta,\bQ,\mathbf{0}}$ with $\beta>0$ and $\bQ_{ij}=\I(0<\|i-j\|_1\leq L)$ for some $L\geq 1$ and $i,j\in \Lambda_n(d)$ (see Section \ref{section:lattice} for precise definitions). Then there exists a $\beta_c(d,L)>0$ such that for all $0\leq \beta<\beta_c(d,L)$ one has
	\begin{align}
	\mathrm{Cov}_{\beta,\bQ,\mathbf{0}}\left(X_i,X_j\right)&\leq \exp\left(-c(\beta,d,L)\|i-j\|_1\right),\label{eqn:correlation_decay}
	\end{align}
	for some $c(\beta,d,L)>0$ depending on $\beta,d,L$.
\end{lemma}

\subsection{Proof of Theorem~\ref{thm:dense_regular}}

\begin{enumerate}
\item[(a)]

\begin{itemize}
\item {\bf High Temperature and Critical Point ($0\leq \beta\leq 1$)}

In this case, note that the average degree $\overline{d}=d_n\gg (\log{n})^{\gamma}$ for any $\gamma>0$ and so we have $\PZ (\bX_S={\bf 1})=2^{-|S|}(1+o(1))$ for any $S$ with $|S|\le 2s$, by Lemma~\ref{lem:smallinsig} ((a) and (b)), where the $o(1)$ term depends on $S$ only through its cardinality. Consequently we have \begin{align*}
\sup_{S_1\cap S_2=\phi,|S_1|=|S_2|=s}\Big|\frac{\PZ (\bX_{S_1}=1,\bX_{S_2}=1)}{\PZ (\bX_{S_1}=1)\PZ (\bX_{S_2}=1)}-1\Big|=o(1).
\end{align*}
The desired conclusion then follows by using Theorem~\ref  {thm:lower_short_range_small_s}, part (I).

\item {\bf Low temperature ($\beta>1$)}

In this case we have $$\PZ (\bX_S={\bf 1}|\bar{\bX}\geq 0)=\lambda_{|S|}^{-1} e^{\beta t|S|}(1+o(1))$$ where $\lambda_{|S|}:=\sum_{\mathbf{b}\in \{-1,1\}^s} \exp(\beta t\sum_{i=1}^s b_i)$ for any set $S$ with $|S|\le 2s$, by Lemma~\ref{lem:smallinsig} (c). Consequently we have $$\sup_{S_1\cap S_2=\phi, |S_1|=|S_2|=s}\Big|\frac{\PZ(\bX_{S_1}=1,\bX_{S_2}=1|\bar{\bX}\geq 0)}{\PZ(\bX_{S_1}=1|\bar{\bX}\geq 0)\PZ(\bX_{S_2}=1|\bar{\bX}\geq 0)}-1\Big|=o(1).$$  Next, set $\Omega_n:=\{\bar{\bX}\geq 0\}$ and note that $\liminf\limits_{n\to\infty}\PZ (\Omega_n)\geq 1/2$ by symmetry. The desired conclusion then follows by invoking Theorem~\ref{thm:lower_short_range_small_s}, part (II).
\\
\end{itemize}

\item[(b)]

\begin{itemize}
\item{\bf High temperature ($0\leq \beta<1$)}
Since $\|\bQ\|_{\infty\rightarrow \infty}=1$ for regular graphs, the upper bounds follows from Theorem \ref{thm:upper}. For the lower bound, note that
\begin{align}\label{eq:trnglect}
\max_{i,j} (\bQ^3)_{ij}=\max_{i,j}d_n^{-3}\sum_{k,l} \bQ_{ik}\bQ_{kl}\bQ_{lj}\lesssim \frac{d_n^2}{d_n^3}\lesssim n^{-1}.
\end{align}
Combining the above observation with Lemma~\ref{lem:slarge} (a), we have \[ \mathrm{Cov}_{\beta,\bQ,\mathbf{0}}(X_i,X_j)=\E_{\beta,\bQ,\mathbf{0}}(X_iX_j)\lesssim \frac{\log n}{n},\]
which immediately gives
\begin{align}\label{eq:varconver}
\sup_{S\in \mathcal{C}_n}\mathrm{Var}_{\beta,\bQ,\mathbf{0}}\left(\sum_{i\in S}X_i\right)\le s+ \sum_{i\ne j}\mathrm{Cov}_{\beta,\bQ,\mathbf{0}}(X_i,X_j)\le s+s^2\frac{ \log n}{n}\lesssim s,
\end{align}
where the last line follows from the fact that $s\leq n^{1-\upsilon}$ for some $\upsilon>0$ which is a standing assumption throughout the paper. For the same reason, we also have 
\[\frac{\log n}{s}\sup_{S_1\ne S_2}\sum_{i\in S_1,j\in S_2}\mathrm{Cov}_{\beta,\bQ,\mathbf{0}}(X_i,X_j)\lesssim s^2\frac{\log n}{s}\frac{\log n}{n}=\frac{s(\log n)^2}{n}=o(1).\]
Thus, invoking Theorem \ref{thm:lower_short_range_large_s}, part (I) with $r_n:=s, r_n':=\frac{s}{\log n}$ then shows that testing is impossible if $\tanh(A)\leq c'\min\Big\{\sqrt{\frac{\log n}{s}},\sqrt{\frac{\log n}{s}}\Big\}=c'\sqrt{\frac{\log n}{s}}$, for some small constant $c'>0$, as desired.

\item{\bf Critical Point ($\beta=1$)}
The upper bound for $s(\log{n})/\sqrt{n}\ll 1$ follows from Theorem \ref{thm:upper} as before. For $s\gtrsim \sqrt{n}/\log{n}$, it suffices to show that there is a sequence of asymptotically powerful tests for $sA\gg n^{1/4}$. Note that by Lemma~\ref{lem:altbeh}, there exists a sequence $k_n:=\delta (n^{-1/4}sA)^{1/3}\to\infty$ for some $\delta>0$, such that $\PG (n^{1/4}\bar{\bX}\geq k_n)\to 1$ as $n\to\infty$. Further by~\cite[Theorem 1.3]{Deb2020}, $\PZ (n^{1/4}\bar{\bX}\geq k_n)\to 0$ as $n\to\infty$. Therefore the test which rejects $\mathrm{H}_0$ if $n^{1/4}\bar{\bX}\geq k_n$ is asymptotically powerful.

For the lower bound, as before, using Lemma \ref{lem:slarge} (a) we have
\[\mathrm{Cov}_{\beta,\bQ,\mathbf{0}}(X_i,X_j)=\E_{\beta,\bQ,\mathbf{0}}(X_iX_j)\lesssim \frac{1}{\sqrt{n}}\Rightarrow
\sup_{S\in \mathcal{C}_n} \sum_{i\in S}\mathrm{Var}_{\beta,\bQ,0}(X_i)\le s+\frac{s^2}{\sqrt{n}},\]
The same bound also gives
\[ \sup_{S_1\ne S_2}\sum_{i\in S_1,j\in S_2}\mathrm{Cov}_{\beta,\bQ,0}(X_i,X_j)\le \frac{s^2}{\sqrt{n}}.\]
We will now consider two cases depending on the value of $s$.

\begin{itemize}
\item{$s\ll \frac{\sqrt{n}}{\log n}$.}
In this case we can invoke Theorem \ref{thm:lower_short_range_large_s} with $r_n:=s$ and $r_n':=\frac{s}{\log n}$ as before to get the detection boundary $\tanh(A)\lesssim \sqrt{\frac{\log n}{s}}$.

\item{$s\gtrsim \frac{\sqrt{n}}{\log n }.$}
In this case setting $\epsilon_n:=\frac{sA}{n^{1/4}}$, $r_n:=s+\frac{s^2}{\sqrt{n}}$ and $r_n':=\frac{s^2}{\epsilon_n \sqrt{n}}$, we note that: \begin{equation}\label{eq:whichmin} \min\left(\sqrt{\frac{\log n}{r_n}},\sqrt{\frac{1}{r_n'}}\right)=\frac{\sqrt{\epsilon_n} n^{1/4}}{s},\end{equation}
where the last equality uses the fact that $s\gtrsim \frac{\sqrt{n}}{\log n}$ which in turn implies \[\sqrt{\frac{\log n}{s}}\ge \frac{\sqrt{\epsilon_n} n^{1/4}}{2s}, \quad \sqrt{\frac{\sqrt{n} \log n}{s^2}}\ge \frac{\sqrt{\epsilon_n} n^{1/4}}{2s},\]
for all large enough $n$. Combining~\eqref{eq:whichmin} with Theorem \ref{thm:lower_short_range_large_s}, part (I) shows that testing is impossible if $sA=o(n^{1/4})$.

\end{itemize}

\item{\bf Low temperature ($\beta>1$)}

As in the low temperature regime for part (a), set $\Omega_n:=\{\bar{\bX}\geq 0\}$. It is an increasing set of probability at least $1/2$ under $\PZ$ (by symmetry). Fix $S\subseteq [n]$ and set $\bmu\equiv \bmu_S(\eta)$ for some $\eta\in [0,2A]$. For any $A\leq c'\sqrt{\log{n}/s}$ for some $c'>0$, note that:
\begin{align*}
&\sum_{i=1}^n \mu_i\leq 2 s A  \leq 2c' \sqrt{s\log{n}}\leq 2c'\sqrt{n\log{n}}\\ & \lVert \bQ \bmu\rVert_{\infty}\leq \max_{i\in [n]}\sqrt{\sum_{j=1}^n \bQ_{ij}^2}\sqrt{\sum_{j=1}^n \mu_j^2}\leq 2A\sqrt{\frac{s}{d_n}}\leq 2c'\sqrt{\frac{\log{n}}{d_n}}, 
\end{align*}
which implies that the conditions for Lemma~\ref{lem:mcontrol}, part (b)(i) hold. Also by choosing $c'>0$ small enough, using $s\geq C\log{n}$ and Lemma~\ref  {lem:fixsol}, we have, without loss of generality, $\limsup\limits_{n\to\infty}\max_{i\in [n]} \beta \sech^2(\beta t+\mu_i)<1$. Therefore, by Lemma~\ref{lem:slarge}, part (b)(i)~and~\eqref{eq:trnglect}, we have $\CG (X_i,X_j|\Omega_n)\lesssim \frac{\log n}{n}$ uniformly in $i,j$. Therefore we can choose $r_n=s$ by the same argument as in~\eqref{eq:varconver}. With $\tilde{S}:=S_1\cup S_2$, observe that:
\begin{align*}
    \mbox{Var}_{\beta,\bQ,\bmu_{\tilde{S}}(\eta)}\left(\sum_{i\in S_1\cup S_2} X_i|\Omega_n\right)&=\mbox{Var}_{\beta,\bQ,\bmu_{\tilde{S}}(\eta)}\left(\sum_{i\in S_1} X_i|\Omega_n\right)+\mbox{Var}_{\beta,\bQ,\bmu_{\tilde{S}}(\eta)}\left(\sum_{i\in S_2} X_i|\Omega_n\right)\\&+\mathrm{Cov}_{\beta,\bQ,\bmu_{\tilde{S}}(\eta)}\left(\sum_{i\in S_1} X_i,\sum_{i\in S_2} X_i|\Omega_n\right)\\ &=2s\sech^2(\beta t+\eta)+O(s\sqrt{\log{n}/d_n}+s^2\log{n}/n),
\end{align*}
where the last line follows from Lemma~\ref{lem:slarge}, part (b), (i) and (ii), and the error term is uniform over $S_1,S_2$ and $\eta\in [0,A]$. Similarly,
$$\mbox{Var}_{\beta,\bQ,\bmu_{S_1}(\eta)}\left(\sum_{i\in S_1} X_i|\Omega_n\right)=s\sech^2(\beta t+\eta)+O(s\sqrt{\log{n}/d_n})$$
and the same conclusion holds with $S_1$ replaced by $S_2$ above, with all error terms being uniform in $S_1,S_2,\eta\in [0,A]$. Consequently,
\begin{align}\label{eq:callater}
&\sup_{\eta\in [0,A]}\sup_{S_1\neq S_2\in\mathcal{C}_n'}\bigg|\mathrm{Var}_{\beta,\bQ,\mathbf{\bmu}_{S_1\cup S_2}(\eta)}\left(\sum_{i\in S_1\cup S_2}X_i|\Omega_n\right)-\mathrm{Var}_{\beta,\bQ,\bmu_{S_1}(\eta)}\left(\sum_{i\in S_1}X_i|\Omega_n\right)\nonumber \\ &-\mathrm{Var}_{\beta,\bQ,\bmu_{S_2}(\eta)}\left(\sum_{i\in S_2}X_i|\Omega_n\right)\bigg|\lesssim O(s\sqrt{\log{n}/d_n}+s^2\log{n}/n)=o(s/\log{n}).
\end{align}
Then, by setting $r_n'=\frac{s}{\log{n}}$ and applying Theorem~\ref{thm:lower_short_range_large_s}, part (II) completes the proof.
\end{itemize}
\end{enumerate}

\subsection{Proof of Theorem~\ref{thm:random_regular}}

Observe that $\lVert \bQ\rVert_{\infty\to\infty}=1$ and $\max\{|\lambda_2(\bQ)|,|\lambda_n(\bQ)|\}=O_p(d_n^{-1/2})$ (see~\cite[Theorem A]{friedman1989second}). Further if $d_n=\Theta(n)$, the result follows from Theorem~\ref  {thm:dense_regular}. Therefore we will assume $d_n=o(n)$ in the rest of the proof. The proof of the whole of part (a), and the critical case $\beta=1$ of part (b), then follows similar to the proof of Theorem~\ref{thm:dense_regular}. In fact, the proofs of the upper bounds for the high and low temperature regimes are also the same as in Theorem~\ref  {thm:dense_regular}. We prove the lower bounds for the high and low temperature regimes below.

\begin{itemize}
\item{\bf High temperature} ($0\leq\beta<1$)

To begin note that the conditions for Lemma~\ref{lem:slarge}, part (a) hold if $\theta>1/2$. Using~\eqref{eq:trnglect}, we therefore have the bound 
\begin{align*}
\sup_{i,j\in S}\CZ (X_i,X_j)\lesssim \frac{1}{d_n},
\end{align*}
and so, if $s\ll \frac{d_n}{\log{n}}$, \begin{align}\label{eq:done}\sup_{S\in \mathcal{C}_n}\VZ(\sum_{i\in S}X_i)\lesssim s+s^2\frac{\log n}{d_n}\lesssim s,\quad \frac{\log n}{s}\sup_{S_1\ne S_2}\sum_{i\in S_1,j\in S_2}\mathrm{Cov}_{\beta,\bQ,\mathbf{0}}(X_i,X_j)\lesssim \frac{s\log{n}}{d_n}=o(1).
\end{align}
Thus invoking Theorem 3 with $r_n=s, r_n'=\frac{s}{\log n}$ gives the desired conclusion. Thus without loss of generality we will assume $s\gtrsim \frac{d_n}{\log n}$ throughout the remainder of the proof.

In the remaining part of the proof we will denote the adjacency matrix any graph   $\mathbb{G}_n=(\mathcal{V}_n,\mathcal{E}_n)$ on $n$-vertices as $\mathbf{G}_n$ and its $i,j$-th element as $\mathbf{G}_n(i,j)$. Also conditioning on a random graph $\mathbb{G}_n$ will imply conditioning w.r.t to the sigma field generated by the random variables involved in $\mathbb{G}_n$ (i.e. the random edges in case of a simple random graph with fixed vertex set). 

Now, if $\mathbb{G}_n$ is a random $d_n$-regular graph on vertices $\{1,\ldots,n\}$, then using~\cite[Theorem 1.5 (b)]{gao2020sandwiching} it follows that $\mathbb{G}_n$ is stochastically dominated by an Erdos-Renyi graph $\widetilde{\mathbb{G}}_n$ (whose corresponding adjacency matrix will be denoted by $\widetilde{\mathbf{G}}_n(i,j)$ for its $i,j$-th element) with parameter $p_n:=\kappa \frac{d_n\log{(n/d_n)}}{n}$ for some fixed $\kappa>0$, and so we have
\begin{align*}
\P((\bQ^3)_{ij}\ge 8\kappa^3 \frac{(\log{(n/d_n)})^3}{n})=\P(\mathbf{G}_n(i,j)^3\ge 8\kappa^3 \frac{d_n^3(\log{(n/d_n)})^3}{n})\le \P(\widetilde{\mathbf{G}}_n^3(i,j)\ge 8n^2 p_n^3).
\end{align*}
Let $F_i$ denote the neighbors of $i$  and note that $\widetilde{\mathbf{G}}_n^3(i,j)=\sum_{k\in F_i,\ell \in F_j}\widetilde{\mathbf{G}}_n(k,\ell)$.
Also we have $|F_i|\sim \mathrm{Bin}(n-1,p_n), |F_j|\sim \mathrm{Bin}(n-1,p_n)$, and given the sets $F_i,F_j$ we further have $\sum_{k\in F_i,\ell \in F_j}\widetilde{\mathbf{G}}_n(k,\ell)$ is stochastically dominated by the $\mathrm{Bin}(|F_i| |F_j|, p_n)$ distribution. Using this we have
\begin{align*}
&\P(\widetilde{\mathbf{G}}_n^3(i,j)\ge 8n^2 p_n^3)\\
\le& 2\P(|F_i|>2 np_n)+\E\Big[ \P(\sum_{k\in F_i,\ell \in F_j}\widetilde{\mathbf{G}}_n(k,\ell)>8n^2p_n^3|F_i,F_j) 1\{\max\{|F_i|,|F_j|\le 2np_n\}\Big]\\
\le &2\P(\mathrm{Bin}(n,p_n)>2np_n)+\P(\mathrm{Bin}(4n^2p_n^2,p_n)>8n^2p_n^3)\le e^{-\delta np_n}+e^{-\delta n^2p_n^3}
\end{align*}
for some $\delta>0$, where the last step uses standard Chernoff bounds for a Binomial distribution. A union bound along with the assumption $\theta>1/2$, which translates to $p_n\geq n^{-\alpha}$ for some $\alpha<1/2$ shows that, on setting $D_n:=\Big\{\max_{i,j\in [n]}(\bQ^3)_{ij}\ge  8 \kappa^3 \frac{(\log{ (n/d_n)})^3}{n}\big\}$ we have
\[\P(D_n)\le n^2 \Big(e^{-\delta np_n}+e^{-\delta n^2p_n^3}\Big)=o(1).\] Also, if $\mathbb{G}_n\in D_n^c$, then for any $i,j\in [n]$ we have the following bound from Lemma~\ref{lem:slarge}, part (a):
\begin{align}\label{eq:bound_1}
\sup_{(i,j)\in  \mathcal{E}_n}\mathrm{Cov}_{\beta,\bQ,\mathbf{0}}(X_i,X_j|\mathbb{G}_n)\lesssim \frac{1}{d_n},\quad
\sup_{(i,j)\notin  \mathcal{E}_n}\mathrm{Cov}_{\beta,\bQ,\mathbf{0}}(X_i,X_j|\mathbb{G}_n)\lesssim \frac{(\log n)^3}{n}.
\end{align}
Let $\mathcal{L}_{n,2s}$ denote the collection of all subsets of $[n]$ of size $2s$. Then for any sets $S\in \mathcal{L}_{n,2s}$ let $E(S,\mathbb{G}_n)$ denote the number of  edges in $\mathbb{G}_n$ within the vertices in $S$. Then we have
\begin{align*}
\P(E(S,\mathbb{G}_n)> 4s^2 p_n)\le &\P(E(S,\widetilde{\mathbb{G}}_n)>4s^2p_n)=\P\left(\mathrm{Bin}({2s\choose 2},p_n)>4s^2p_n\right)\le e^{-\delta s^2p_n}.
\end{align*}
Setting $E_n:=\Big\{\sup_{S\in \mathcal{L}_{n,2s}}E(S,\mathbb{G}_n)> 4s^2p_n\}$, a union bound then gives
\[\P(E_n)\le {n\choose 2s} e^{-\delta s^2p_n}\le n^{2s} e^{-\delta s^2 p_n}=o(1),\]
where we use the bound \[s^2p_n\ge s\cdot \frac{d_n}{\log n}\cdot \frac{\kappa d_n\log{ (n/d_n)}}{n}=s\cdot \frac{d_n^2}{n\log n}\gg 2s\log{n}.\]
Combining we have $\P(D_n\cup E_n)=o(1)$. For $\mathbb{G}_n\in D_n^c\cap E_n^c$, using the bound \eqref{eq:bound_1} gives
\begin{align*}
\sup_{S\in \mathcal{C}_n}\VZ \left(\sum_{i\in S}X_i|\mathbb{G}_n\right)\lesssim  s+\frac{s^2 p_n}{d_n}+s^2 \frac{(\log n)^3}{n}\lesssim s,\\
\frac{\log n}{s}\sup_{S_1\ne S_2}\sum_{i\in S_1,j\in S_2}\CZ (X_i,X_j|\textcolor{black}{\mathbb{G}_n})\lesssim \frac{\log n}{s}\Big(\frac{s^2 p_n}{d_n}+s^2 \frac{(\log n)^3}{n}\Big)=o(1).
\end{align*}
Thus again we have verified \eqref{eq:done}, and so invoking Theorem 3 with $r_n=s, r_n'=\frac{s}{\log n}$ gives the desired conclusion as before using Theorem~\ref  {thm:lower_short_range_large_s}, part (I).

\item{\bf Low temperature}

Set $\Omega_n:=\{\overline{\bX}\geq 0\}$. The proof is similar to Theorem~\ref{thm:dense_regular}, part (b) for the low temperature regime. Without loss of generality, we assume $d_n=o(n)$ as before. Note that, on the set $D_n^c$ defined above we have by the same calculation as in the high temperature regime:
\begin{align}\label{eq:ltembbd}
\sup_{(i,j)\in \mathcal{E}_n} \CG (X_i,X_j|\bar{\bX}\geq 0,\mathbb{G}_n)\lesssim \frac{1}{d_n},\sup_{(i,j)\notin \mathcal{E}_n} \CG (X_i,X_j|\bar{\bX}\geq 0,\mathbb{G}_n)\lesssim \frac{(\log n)^3}{n},
\end{align}
by Lemma~\ref{lem:slarge}, part (b) where we have used the fact that $\theta>2/3$. Using~\eqref{eq:ltembbd} and $s\ll \frac{d_n}{\log{n}}$, we get:
\begin{equation}\label{eq:varbdhere}
\sup_{\eta\in [0,2A]}\sup_{S\in\mathcal{C}_n}\mbox{Var}_{\beta,\bQ,\bmu_S(\eta)} \left(\sum_{i\in S} X_i|\bar{\bX}\geq 0, \mathbb{G}_n\right)\lesssim s+\frac{s^2\log{n}}{d_n}\lesssim s.
\end{equation}
Also by the same calculation as in~\eqref{eq:callater}, we see that 
\begin{align}\label{eq:covbdhere}
&\sup_{\eta\in [0,A]}\sup_{S_1\neq S_2\in\mathcal{C}_n'}\bigg|\mathrm{Var}_{\beta,\bQ,\mathbf{\bmu}_{S_1\cup S_2}(\eta)}\left(\sum_{i\in S_1\cup S_2}X_i|\Omega_n, \mathbb{G}_n\right)-\mathrm{Var}_{\beta,\bQ,\bmu_{S_1}(\eta)}\left(\sum_{i\in S_1}X_i|\Omega_n, \mathbb{G}_n\right)\nonumber \\ &-\mathrm{Var}_{\beta,\bQ,\bmu_{S_2}(\eta)}\left(\sum_{i\in S_2}X_i|\Omega_n, \mathbb{G}_n\right)\bigg|\lesssim s\sqrt{\frac{\log{n}}{d_n}}+\frac{s^2}{d_n}=o_{p}\bigg(\frac{s}{\log{n}}\bigg).
\end{align}
Therefore by choosing $r_n=s$, $r_n'=s/\log{n}$ and using Theorem~\ref{thm:lower_short_range_large_s}, part (II) then completes the proof. We can therefore assume $s\gtrsim \frac{\log{n}}{d_n}$. In this case, the same conclusions as in~\eqref{eq:varbdhere}~and~\eqref{eq:covbdhere} follow as in the proof of the high temperature regime if we further restrict to $G_n\in D_n^c\cap E_n^c$. We omit the details for brevity.
%
\end{itemize}

\subsection{Proof of Theorem~\ref{thm:erdos_renyi}}
The proof goes through exactly as the proof of Theorem~\ref  {thm:random_regular}, with the only change being that now $\mathbb{G}_n$ is not regular but ``approximately regular". In order to carry out the same proof, we need to show that $\mathbb{G}_n$ (having adjacency matrix $\mathbf{G}_n$ and $\bQ=\mathbf{G}_n/\overline{d}$) with and $\theta>0$, satisfies:
\begin{enumerate}
    \item[(1).] $\max_{i\in [n]}\Big|\frac{d_i}{\overline{d}}-1\Big|=O_p\Big(\sqrt{\frac{\log n}{\overline{d}}}\Big)$
    \item[(2).] $\lVert \bQ\rVert_{\infty\to\infty}=O_p(1)$.
    \item[(3).] $\max\{|\lambda_2(\bQ)|,|\lambda_n(\bQ)|\}=o_p(1)$.
\end{enumerate}

Here we as usual have defined $d_i=\sum_{j=1}^n \mathbf{G}_{n}(i,j)$ to be the degree of the $i^{\mathrm{th}}$ vertex and $\overline{d}$ the average degree of the graph. As $d_1\sim \mathrm{Bin}(n-1,p_n)$, a standard Chernoff's inequality yields that $\sqrt{np_n}\Big|\frac{d_1}{np_n}-1\Big|=O_p(1)$ and a union bound then yields $\sqrt{\frac{np_n}{\log{n}}}\max_{i\in [n]} \Big|\frac{d_i}{np_n}-1\Big|=O_p(1)$. A similar argument also shows that $\sqrt{np_n}\Big|\frac{\bd}{np_n}-1\Big|=O_p(1)$. Combining these observations yields:
\begin{align*}
    \sqrt{\frac{\overline{d}}{\log{n}}}\max_{i\in [n]} \bigg|\frac{d_i}{\overline{d}}-1\bigg|\lesssim \frac{np_n}{\sqrt{\overline{d}\log{n}}}\left(\max_{i\in [n]}\bigg|\frac{d_i}{np_n}-1\bigg|+\bigg|\frac{\overline{d}}{np_n}-1\bigg|\right)=O_p(1),
\end{align*}
which establishes (1). Note that (2) follows from (1), and (3) follows from~\cite[Theorem 1.1]{Feige2005}.

\subsection{Proof of Theorem \ref{thm:lattice}}
In this proof we follow the notation introduced in Section \ref{section:lattice}.

{{\bf (a) $s\ge C\log n$}}

For this part, it suffices verify the conditions of Theorem \ref{thm:lower_short_range_large_s}  for some large constant $C>0$. To this end, 
we first define a sub-collection $\mathcal{C}_n'$ of $\mathcal{C}_n$ as follows.  Throughout we assume that $s^{1/d}$ and $n^{1/d}$ are integers for the sake of notational convenience. The analyses works verbatim otherwise by working with the corresponding ceiling functions. Also assume without loss of generality that $3s^{1/d}$ divides $n^{1/d}$. First, let $\mathcal{C}_n''$ be the class of disjoint sub-cubes of $\Lambda_n(d)$ obtained by translating along each axis (by $3s^{1/d}$ in each direction each time) the cube of side lengths $3s^{1/d}$ from the bottom left corner of $\Lambda_n(d)=[-n^{1/d},n^{1/d}]^d\cap \mathbb{Z}^d$. Consequently, subdivide each cube in $\mathcal{C}_n'$ into $3^d$ cubes of side length $s^{1/d}$ each {\color{black} and take the} center sub-cube of each cube in $\mathcal{C}_n''$ to be elements of our class $\mathcal{C}_n'$. It is easy to see that $|\mathcal{C}_n'|=|\mathcal{C}_n''|=(2/3)^d\frac{n}{s}$ and also that $\min\limits_{i\in S_1,j\in S_2:\atop S_1\neq S_2\in \mathcal{C}_n}\|i-j\|_1\geq 4s^{1/d}.$

First note that $\|\bQ\|_{\infty\rightarrow \infty}=2dL$. Also note that by Lemma \ref{lemma:correlation_decay}, for all $0\leq \beta<\beta_c(d,L)$ one has
\begin{align}
\mathrm{Cov}_{\beta,\bQ,\mathbf{0}}\left(X_i,X_j\right)&\leq \exp\left(-c(\beta,d,L)\|i-j\|_1\right),
\end{align}
for some $c(\beta,d,L)>0$ depending on $\beta,d,L$. Therefore for any $S\in \mathcal{C}_n'$ we have
\begin{align*}
\mathrm{Var}_{\beta,\bQ,\mathbf{0}}(\sum_{i\in S}X_i)&\leq \sum_{i,j\in S}\exp\left(-c(\beta,d,L)\|i-j\|_1\right)\\
&\leq \sum_{l=0}^{ds^{1/d}}\sum_{i,j\in S:\|i-j\|_1= l} \exp\left(-c(\beta,d,L)\|i-j\|_1\right)\leq C's,
\end{align*}
for some constant $C'>0$.
Also for any $S_1\neq S_2\in \mathcal{C'}$
\begin{align*}
\sum_{i\in S_1, j\in S_2}\mathrm{Cov}_{\beta,\bQ,\mathbf{0}}(X_i,X_j)&\leq \sum_{i\in S_1, j\in S_2}\exp\left(-c(\beta,d,L)\|i-j\|_1\right)\\
&\leq s^2\exp\left(-4c(\beta,d,L)s^{1/d}\right).
\end{align*}


This completes the verification of the conditions of Theorem \ref{thm:lower_short_range_large_s}, thus verifying part (a).

{{\bf (b) $s\le c\log n$}}

It suffices to verify the conditions of Theorem \ref{thm:lower_short_range_small_s} 
for some small constant $c>0$. To this end, 
we again define a sub-collection of $\mathcal{C}_n$ as follows.  First, let $\mathcal{C}_n'$ be the class of disjoint sub-cubes of $\Lambda_n(d)$ obtained by translating along each axis (by $\log{n}$ in each direction each time) the cube of side lengths $s^{1/d}$ from the bottom left corner of $\Lambda_n(d)=[-n^{1/d},n^{1/d}]^d\cap \mathbb{Z}^d$.  It is easy to see that $|\mathcal{C}_n'|\gtrsim\left(\frac{n}{s+\log^d n}\right)\gtrsim \frac{n}{\log^d n}$, and also  $\min\limits_{i\in S_1,j\in S_2:\atop S_1\neq S_2\in \mathcal{C}_n'}\|i-j\|_1\geq \log{n}.$ Consider now any two $S_1\neq S_2\in \mathcal{C}_n'$ and consider the ratio 
\begin{align*}
  \frac{\P_{\beta,\bQ,\mathbf{0}}(\bX_{S_1}=1,\bX_{S_2}=1)}{\P_{\beta,\bQ,\mathbf{0}}(\bX_{S_1}=1)\P_{\beta,\bQ,\mathbf{0}}(\bX_{S_2}=1)}.  
\end{align*}
To analyze this ratio, let $\tilde{\P}$ denote the Edward-Sokal coupling measure between the Ising model $\P_{\beta,\bQ,\mathbf{0}}$ and the corresponding random cluster model (see e.g. \cite{grimmett2006random,duminil2017lectures}). In the following argument, for any two sets $A,B\in \Lambda_n(d)$ we denote $A\leftrightarrow B$ (respectively $A\not \leftrightarrow B$) to denote the event that there is an open path between the sets $A$ and $B$ (respectively there is not open path between $A$ and $B$). Then
\begin{align*}
    \P_{\beta,\bQ,\mathbf{0}}(\bX_{S_1}=1,\bX_{S_2}=1)&=\tilde{\P}(\bX_{S_1}=1,\bX_{S_2}=1)\\
    &=\tilde{\P}(\bX_{S_1}=1,\bX_{S_2}=1,S_1\leftrightarrow S_2)+\tilde{\P}(\bX_{S_1}=1,\bX_{S_2}=1,S_1\not \leftrightarrow S_2)
\end{align*}
Now, since $0\leq \beta<\beta_c(d)$ and $\min\limits_{i\in S_1,j\in S_2:\atop S_1\neq S_2\in \mathcal{C}_n'}\|i-j\|_1\geq \log{n}$, there exists a constant $\rho>0$ such that
\begin{align}
    \tilde{\P}(\bX_{S_1}=1,\bX_{S_2}=1,S_1 \leftrightarrow S_2)\leq \tilde{\P}(S_1 \leftrightarrow S_2) \leq \exp(-\rho\log{n}).\label{eqn:random_cluster_exp_decay}
\end{align}
Similarly, since under the Edward-Sokal coupling disjoint clusters are assigned spins independent of one another, we have
\begin{align*}
    \tilde{\P}(\bX_{S_1}=1,\bX_{S_2}=1,S_1\not\leftrightarrow S_2)
    &=\tilde{\P}(\bX_{S_1}=1,\bX_{S_2}=1|S_1\not \leftrightarrow S_2)\tilde{P}(S_1\not\leftrightarrow S_2)\\
    &=\tilde{\P}(\bX_{S_1}=1|S_1\not \leftrightarrow S_2)\tilde{\P}(\bX_{S_2}=1|S_1\not \leftrightarrow S_2)\tilde{\P}(S_1\not \leftrightarrow S_2)\\
\end{align*}
Moreover, by FKG Inequality we have
\begin{align}
{\P_{\beta,\bQ,\mathbf{0}}}(\bX_{S_1}=1){\P_{\beta,\bQ,\mathbf{0}}}(\bX_{S_2}=1)\geq 2^{-2s}.\label{eqn:prob_larger_than_indep}
\end{align}
Therefore, 
\begin{align*}
    \frac{\P_{\beta,\bQ,\mathbf{0}}(\bX_{S_1}=1,\bX_{S_2}=1)}{\P_{\beta,\bQ,\mathbf{0}}(\bX_{S_1}=1)\P_{\beta,\bQ,\mathbf{0}}(\bX_{S_2}=1)}
    &=\frac{\tilde{\P}(\bX_{S_1}=1|S_1\not \leftrightarrow S_2)\tilde{\P}(\bX_{S_2}=1|S_1\not \leftrightarrow S_2)\tilde{\P}(S_1\not \leftrightarrow S_2)}{\P_{\beta,\bQ,\mathbf{0}}(\bX_{S_1}=1)\P_{\beta,\bQ,\mathbf{0}}(\bX_{S_2}=1)}\\
    &+\frac{\tilde{\P}(\bX_{S_1}=1,\bX_{S_2}=1,S_1 \leftrightarrow S_2)}{\P_{\beta,\bQ,\mathbf{0}}(\bX_{S_1}=1)\P_{\beta,\bQ,\mathbf{0}}(\bX_{S_2}=1)}. \label{eqn:ratio_lessthan_logn}
\end{align*}
The second term in the display above is smaller than $\exp(-\rho\log{n}+s\log 4)$ by \eqref{eqn:random_cluster_exp_decay} and  \eqref{eqn:prob_larger_than_indep}. Therefore, for $s\le c\log n$ with $c$ small enough, this converges to $0$, and we only need to show that the first term above is $1+o(1)$ uniformly in $S_1\neq S_2$. To this end, note that
\begin{align*}
  \frac{\tilde{\P}(\bX_{S_1}=1|S_1\not \leftrightarrow S_2)\tilde{\P}(\bX_{S_2}=1|S_1\not \leftrightarrow S_2)\tilde{\P}(S_1\not \leftrightarrow S_2)}{\P_{\beta,\bQ,\mathbf{0}}(\bX_{S_1}=1)\P_{\beta,\bQ,\mathbf{0}}(\bX_{S_2}=1)}&=T_1+T_2+T_3,  
\end{align*}
where 
\begin{align*}
    T_1&=\frac{\tilde{\P}(\bX_{S_1}=1)\tilde{\P}(\bX_{S_2}=1)}{\tilde{\P}(S_1\not \leftrightarrow S_2)\P_{\beta,\bQ,\mathbf{0}}(\bX_{S_1}=1)\P_{\beta,\bQ,\mathbf{0}}(\bX_{S_2}=1)}
\end{align*}
\begin{align*}
    T_2&=-\frac{\tilde{\P}(\bX_{S_1}=1,S_1\leftrightarrow S_2)\tilde{\P}(X_{S_2}=1)+\tilde{\P}(\bX_{S_2}=1,S_1\leftrightarrow S_2)\tilde{\P}(X_{S_1}=1)}{\tilde{\P}(S_1\not \leftrightarrow S_2)\P_{\beta,\bQ,\mathbf{0}}(\bX_{S_1}=1)\P_{\beta,\bQ,\mathbf{0}}(\bX_{S_2}=1)}
\end{align*}

\begin{align*}
T_3&=\frac{\tilde{\P}(\bX_{S_1}=1,S_1\leftrightarrow S_2)\tilde{\P}(\bX_{S_2}=1,S_1\leftrightarrow S_2)}{\tilde{\P}(S_1\not \leftrightarrow S_2)\P_{\beta,\bQ,\mathbf{0}}(\bX_{S_1}=1)\P_{\beta,\bQ,\mathbf{0}}(\bX_{S_2}=1)}    
\end{align*}
Once again, it is easy to see from \eqref{eqn:random_cluster_exp_decay} and  \eqref{eqn:prob_larger_than_indep} that there exists a constant $\tilde{\rho}$ such that for large enough $n$ one has $|T_2+T_3|\lesssim  C_1\exp(-\tilde{\rho}\log{n}+s\log{4})$. For $T_1$, note that by definition of the coupling we have
\begin{align*}
T_1&=\frac{{\P_{\beta,\bQ,\mathbf{0}}}(\bX_{S_1}=1)\P_{\beta,\bQ,\mathbf{0}}(\bX_{S_2}=1)}{\tilde{\P}(S_1\not \leftrightarrow S_2)\P_{\beta,\bQ,\mathbf{0}}(\bX_{S_1}=1)\P_{\beta,\bQ,\mathbf{0}}(\bX_{S_2}=1)}\\
&=\frac{1}{\tilde{\P}(S_1\not \leftrightarrow S_2)}.
\end{align*}
From \eqref{eqn:random_cluster_exp_decay} we immediately have $T_1=1+o(1)$ uniformly in $S_1\neq S_2\in \mathcal{C}_n'$. This completes the verification of the condition of Theorem \ref{thm:lower_short_range_small_s} for $0\leq \beta<\beta_c(d)$.

\section{Proofs of Auxiliary Lemmas}\label{sec:comres}



This section is devoted to proving Lemmas~\ref{lem:slarge},~\ref{lem:smallinsig},~and~\ref{lem:altbeh}.


In the sequel, we will use $d_{\max}$ to represent the maximum degree of the graph  $\mathbb{G}_n$ and $\mathbf{I}_m$ for the $m\times m$ identity matrix. Finally, throughout we let $\alpha_n=\sqrt{\frac{\log{n}}{\overline{d}}}$, for $\bX\sim \P_{\beta,\bQ,\bmu}$ let $m_i=m_i(\bX):=\sum_{j=1}^n \bQ_{ij}X_j$, $\bar{\mathbf{m}}=\sum_{i=1}^n m_i/n$ , and use the letter $t$ to denote the non-negative root of $x=\tanh(\beta x)$ for $\beta>1$.

Our first result yields a sharp control on the tail behavior of $m_i,i\geq 1$ -- which serves as a crucial building block for proving Lemmas~\ref{lem:slarge},~\ref{lem:smallinsig},~and~\ref{lem:altbeh}.

\begin{lemma}\label{lem:mcontrol}
	Suppose that $\bQ$ is the scaled adjacency matrix of a graph, such that $\max_{1\le i\le n}\Big|\frac{d_i}{\overline{d}}-1\Big|\to  0$, and let $\lambda\ge 1$. Then we have the following conclusions:
	\begin{enumerate}
	\item[(a)] For $0\leq \beta<1$ we have:
	\begin{align*}
	\log{\PZ\Big(\max_{i\in [n]}|m_i|>\lambda\alpha_n\Big)}\lesssim -\lambda^2,
	\end{align*}
	for all large enough $n$.

	\item[(b)] Let $\beta>1$, and $\max_{i\in [n]}\Big|\frac{d_i}{\overline{d}}-1\Big|\lesssim \alpha_n.$ 
	
	(i) If $\mathbb{G}_n$ satisfies  $\bd\gg \sqrt{n \log n}$ and
	
	\begin{equation}\label{eq:connected}
\limsup\limits_{n\to\infty}\lambda_2(\bQ)<1,
\end{equation} 
	and $\bmu$ satisfies  
	\begin{align}\label{eq:mu_cond}
	\sum_{i=1}^n\mu_i\lesssim\sqrt{n\log n},\quad \lVert \bQ\bmu\rVert_\infty\lesssim \alpha_n,
	\end{align}
	 then we have
	\begin{align*}
	\log{\PG\left(\max_{i\in [n]} |m_i-t|\geq \lambda \alpha_n,\bar{\bX}\geq 0\right)}\lesssim -\lambda^2,
	\end{align*}
	for all large enough $n$.
	
	
(ii) 
	If  $\mathbb{G}_n$ satisfies $\overline{d}\ge n^\gamma$ for some $\gamma>0$, and $\max_{2\le i\le n}|\lambda_i(\bQ)| \to 0$, and $\boldsymbol{\mu}$ satisfies \eqref{eq:mu_cond}, then we  have:
	\begin{align*}
	\log{\PZ\left(\max_{i\in [n]} |m_i-t|\geq \lambda \alpha_n,\bar{\bX}\geq 0\right)}\lesssim -\min{(\lambda^2,\bd^{1-\gamma})},
	\end{align*}
	for all large enough $n$.

	\item[(c)] For $\beta=1$, assume that $\max_{i\in [n]}\Big|\frac{d_i}{\overline{d}}-1\Big|\lesssim \alpha_n$.
	
(i) Then we have
	\begin{align*}
	\log\PZ \left(\max_{i\in [n]} |m_i|>\lambda \alpha_n^{1/3}\right)\lesssim -\lambda^2
	\end{align*}
	for all large enough $n$.
	
	(ii) If $\mathbb{G}_n$ satisfies \eqref{eq:connected}, and $\bmu\in (\mathbb{R}^+)$ satisfies  $\lVert \bQ\bmu\rVert_\infty \lesssim \alpha_n$,  then we have:
	\begin{align*}
	\log{\PG\left(\max_{i\in [n]} |m_i-\bar{\mm}|>\lambda \frac{(\log n)^{3/2}}{\overline{d}}+\lambda \sqrt{\bar{\bmu}}\right)}\lesssim -\lambda^2,
	\end{align*}
	for all large enough $n$.

	\end{enumerate}
\end{lemma}

\begin{proof}
	Part (a) follows directly from part (a) of~\cite[Lemma 2.3]{Deb2020}. Part (c)(i) follows by combining \cite[Equations (4.9), (4.10)]{Deb2020}. Here we prove the remaining parts (b)(i) (ii), and (c) (ii).


	\emph{Part (b)}(i). Without loss of generality we can assume $\lambda \alpha_n\le 1$, as otherwise the bound is trivial on noting that $\lambda \alpha_n> 1\gtrsim \max_{i\in [n]}|m_i|.$
	 We now claim that
	\begin{align}\label{eq:al}
	\log\PG (A_{n,\lambda})\lesssim -\lambda^2,\quad A_{n,\lambda}:=\left\{\max_{i\in [n]}|m_i-\sum_{j=1}^n\bQ_{ij}\tanh(\beta m_j)|> \lambda \alpha_n\right\}.
	\end{align}
	Indeed, this follows on using Lemma \ref{lemma:chatterjee} along with the bound $\lVert \bQ\bmu \rVert_\infty\lesssim \alpha_n.$ 
	
	A two term Taylor expansion of $\tanh(\beta m_j)$ gives
	\[\Bigg| \sum_{j=1}^n\bQ_{ij}\tanh(\beta m_j)-\frac{d_i}{\overline{d}}t-\beta (1-t^2)\sum_{j=1}^n \bQ_{ij}(m_j-t)\Bigg|\lesssim \sum_{j=1}^n\bQ_{ij}(m_j-t)^2. \]
	and so on the set $A_{n,\lambda}^c$ we have
	\begin{align}\label{eq:parksandrec}
	\notag&\max_{i\in [n]}|m_i-t|\Big[1-\beta(1-t^2)\max_{i\in [n]}d_i/\bd\Big]\lesssim \lambda \alpha_n+\max_{i\in [n]}\left[\sum_{j=1}^n \bQ_{ij}(m_j-t)^2+\Bigg|\frac{d_i}{\overline{d}}-1\Bigg|\right]\\
	\Rightarrow &\max_{i\in [n]}|m_i-t| \lesssim \max_{i\in [n]}\sum_{j=1}^n Q_{ij}(m_j-t)^2+ \lambda \alpha_n,
	\end{align}
	where the last line uses the fact that $\max_{i\in [n]}\frac{d_i}{\overline{d}}\to 1$, and $\beta(1-t^2)<1$. 
	We now claim that for every $\varepsilon>0$ we have
	\begin{align}\label{eq:b}
	\log \PG(B_{n,\varepsilon})\lesssim -\frac{1}{\alpha_n^2},\quad B_{n,\varepsilon}:=\Big\{\max_{i\in [n]}|m_i-t|>\varepsilon\Big\}.
	\end{align}
	Given \eqref{eq:b}, choosing $\varepsilon>0$ small enough, on the set $A_{n,\lambda}^c\cap B_{n,\varepsilon}^c$, using \eqref{eq:parksandrec} we have
	\[ \max_{i\in [n]}|m_i-t|\lesssim \lambda \alpha_n+\max_{i\in [n]}|m_i-t|^2\Rightarrow \max_{i\in [n]}|m_i-t|\lesssim \lambda \alpha_n.\]
	Combining the above observations, we get
\begin{align}\label{eq:combine2}
\log \PG (\max_{i\in [n]}|m_i-t|\gtrsim \lambda \alpha_n)\le \log \Big[\PG (A_{n,\lambda})+\PG (B_{n,\varepsilon})\Big]\lesssim -\min\Big[\lambda^2, \frac{1}{\alpha_n^2}\Big]
\end{align}
 using \eqref{eq:al} and \eqref{eq:b}. The desired conclusion follows from this on recalling that $\lambda \alpha_n\le 1$.

It thus suffices to verify \eqref{eq:b}. To this effect, use \eqref{eq:al} with $\lambda \alpha_n=\varepsilon$ to note that
\begin{align}\label{eq:a_r}
\log \PG (A_{n, \varepsilon/\alpha_n})=\log \PG \left(\max_{i\in [n]}|m_i-\sum_{j=1}^n\bQ_{ij}\tanh(\beta m_j)|>\varepsilon\right)\lesssim -\frac{1}{\alpha_n^2}.
\end{align}
We now claim that for any $0<M<\infty$ and for all $L\ge M\alpha_n^{-2}\log n$ (i.e., $L\geq M\bd$), we have
	\begin{align}\label{eq:c}
	\log \PG (C_{n,L})\lesssim- L,  \text{ where }C_{n,L}:=\Big\{\sum_{i=1}^n(m_i-t)^2>L, \bar{\bX}\geq 0\Big\}.
	\end{align}

	On the set $C_{n,\varepsilon\bd }^c$ we have
	\begin{align*}
	\sum_{j=1}^n\bQ_{ij}(m_j-t)^2\le \frac{1}{\overline{d}}\sum_{j=1}^n(m_j-t)^2\le \varepsilon ,
	\end{align*}
	which along with \eqref{eq:parksandrec} implies that there exists constants $M_1>0$ such that:
	\begin{align*}
	\ & \log \PG (\max_{i\in [n]}|m_i-t|>\varepsilon,\bar{\bX}\geq 0)\\
	\le &\log \PG\left(\max_{i\in [n]}\sum_{j=1}^n \bQ_{ij}(m_j-t)^2\geq  M_1\varepsilon,\bar{\bX}\geq 0\right)\\
	\le& \log \max\left[\PG(A_{n,M_1\varepsilon/\alpha_n}),\PG(C_{n,\varepsilon\overline{d}})\right]\lesssim -\min\left[\frac{1}{\alpha_n^2},\overline{d}\right]=-\alpha_n^{-2},
	\end{align*}
	where the last inequality uses \eqref{eq:a_r} and \eqref{eq:c}. This
	verifies \eqref{eq:b}. Finally,  \eqref{eq:c}
 follows using \eqref{eq:mu_cond} to note that  $\sup_{\mathbf{x}\in [-1,1]^n} |\sum_{i=1}^n \mu_ix_i|\le \sum_{i=1}^n\mu_i\lesssim \sqrt{n\log n},$ and consequently for $\delta$ small enough we have
	\begin{align}\label{eq:changemes}
	\log\EG \left[\exp\left(\delta\sum_{i=1}^n (m_i-t)^2\right)|\bar{\bX}\geq 0\right]\leq& \sqrt{n\log n}+\log \EZ\left[\exp\left(\delta\sum_{i=1}^n (m_i-t)^2\right)|\bar{\bX}\geq 0\right]\nonumber \\
	\lesssim &\sqrt{n\log n}+\sum_{i=1}^n\Big(\frac{d_i}{\overline{d}}-1\Big)^2+\frac{n}{\overline{d}}=o(\overline{d}), 
	\end{align}
where the last line uses~\cite[Lemma 2.2 part (b)]{Deb2020}.

	\emph{Part (b)}(ii) 
	
	 We begin by claiming that for every $\varepsilon>0$ we have
	\begin{align}\label{eq:new_claim}
	\log \PG(\max_{i\in [n]}|m_i-t|>\varepsilon)\lesssim -\overline{d}^{1-\gamma}.
	\end{align}
	Given this claim, note that \eqref{eq:new_claim} is the analogue to \eqref{eq:b} above, which is the only place where we use the fact that $\overline{d}\gg \sqrt{n\log n}$ in the proof of part (c)(i). Thus, following the above proof for the derivation of \eqref{eq:combine2} gives
	\[\log \PG(\max_{i\in [n]}|m_i-t|>\lambda \alpha_n)\lesssim \max\Big[\log \PG(A_{n, \lambda}),\log \PG(B_{n,\varepsilon})\Big]=-\min\Big[\lambda^2, \overline{d}^{1-\gamma}\Big],\]
	as desired.
It thus remains to verify \eqref{eq:new_claim}. 

To this effect, a one term Taylor's series expansion of $\tanh(\beta m_j)$ gives
\[\sum_{j=1}^n\bQ_{ij}\tanh(\beta m_j)=t\frac{d_i}{\overline{d}}+\beta\sum_{j=1}^n\bQ_{ij}(m_j-t)\text{sech}^2(\beta \xi_j)=t\frac{d_i}{\overline{d}}+\beta(1-t^2)\frac{1}{n} \sum_{j=1}^n(m_j-t)+\sum_{j=1}^n\mathbf{D}^{(n)}_{ij}(m_j-t),\]
where $\mathbf{D}^{(n)}_{ij}:=\beta \left[\bQ_{ij}\text{sech}^2(\beta \xi_j)-\frac{1-t^2}{n}\right]$, for some $\xi_j$ lying between $m_j$ and $t$. On the set $A_{n,\sqrt{\overline{d}^{1-\gamma}}}^c\cap C^c_{n,n\log n/\overline{d}^{\gamma}}$ this gives
\begin{align*}|(m_i-t)-\sum_{j=1}^n\mathbf{D}^{(n)}_{ij}(m_j-t)|\lesssim &\Big|\frac{d_i}{\overline{d}}-1\Big|+\Big|\frac{1}{n}\sum_{i=1}^n(m_i-t)\Big|+\sqrt{\frac{\log n}{\overline{d}^{\gamma}}}\\
\le &\Big|\frac{d_i}{\overline{d}}-1\Big|+\sqrt{\frac{1}{n}\sum_{i=1}^n(m_i-t)^2}+  \sqrt{\frac{\log n}{\overline{d}^{\gamma}}}\le 3\sqrt{\frac{\log n}{\overline{d}^{\gamma}}}
\end{align*}
for all $n$ large enough.
With $K$ denoting the implied constant in the display above we have
\begin{align}\label{eq:recur_1}
\max_{i\in [n]}\Bigg|(m_i-t)-\sum_{j=1}^n\mathbf{D}^{(n)}_{ij}(m_j-t)\Bigg|\le K\sqrt{\frac{\log n}{\overline{d}^{\gamma}}}
\end{align}
Consequently, for every $\ell\ge 1$ we have
\begin{align*}
&\Bigg|\sum_{j=1}^n(\mathbf{D}^{(n)}_{ij})^\ell(m_j-t)-\sum_{j=1}^n \mathbf{D}^{(n)}_{ij}\sum_{k=1}^n(\mathbf{D}^{(n)}_{jk})^\ell(m_k-t)\Bigg|\\
\le &\max_{j\in [n]}\Bigg|(m_j-t)-\sum_{k=1}^n\mathbf{D}^{(n)}_{jk}(m_j-t)\Bigg|\max_{i\in [n]}\Bigg|\sum_{j=1}^n(\mathbf{D}^{(n)}_{ij})^\ell\Bigg| 
\le (2\beta)^\ell K\sqrt{\frac{\log n}{\overline{d}^{\gamma}}},
\end{align*}
where the last inequality uses the bound $\max_{i\in [n]}\sum_{j=1}^n|\mathbf{D}^{(n)}_{ij}|\le 2\beta$ for all $n$ large enough. Combining the last two displays show that for any $\ell\ge 1$ we have
\[\max_{i\in [n]}\Bigg|(m_i-t)-\sum_{j=1}^n(\mathbf{D}^{(n)}_{ij})^\ell(m_j-t)\Bigg|\le  K\overline{d}^{-\gamma} \sum_{r=0}^{\ell-1}(2\beta)^r\leq (2\beta)^\ell K\sqrt{\frac{\log n}{\overline{d}^{\gamma}}} ,  \]
and so
\begin{align}\label{eq:bound_11}
\max_{i\in [n]}|m_i-t|\le \lVert (\mathbf{D}^{(n)})^\ell\rVert_\infty \max_{i\in [n]}|m_i-t| +(2\beta)^\ell K\overline{d}^{-\gamma}\lesssim  \lVert (\mathbf{D}^{(n)})^\ell\rVert_\infty+(2\beta)^\ell \sqrt{\frac{\log n}{\overline{d}^{\gamma}}}.
\end{align}
We now claim that on the set $A_{n,\sqrt{\overline{d}^{1-\gamma}}}^c\cap C^c_{n,n\log n/\overline{d}^{\gamma}}$ we have
\begin{align}\label{eq:op}
\lim_{n\rightarrow\infty}\lVert \mathbf{D}^{(n)}\rVert_2=0.
\end{align}
Given \eqref{eq:op}, using spectral theorem write $(\mathbf{D}^{(n)})^{\top}\mathbf{D}^{(n)}=\sum_{i=1}^n\lambda_i {\bf p}_i{\bf p}_i$, where $\max_{i\in [n]}|\mu_i|=o(1)$, and so
\[\Bigg|\left\{\left((\mathbf{D}^{(n)})^{\top}\mathbf{D}^{(n)}\right)^{\ell}\right\}_{ij}\Bigg|=\Bigg|\left(\sum_{k=1}^n\lambda_i^\ell {\bf p}_k {\bf p}_k'\right)_{ij}\Bigg|\le \max_{k\in [n]}|\lambda_k|^{\ell} .\]
This immediately shows that setting $\ell=\delta \log n$ with $\delta=\frac{\gamma }{2\log (2\beta)}$, using \eqref{eq:bound_11} we have 
\[\max_{i\in [n]}|m_i-t|\lesssim n\max_{i\in [n]}|\lambda_i|^{\ell}+(2\beta)^{\delta \log n}\sqrt{\frac{\log n}{\overline{d}^{\gamma}}}\le \sqrt{\frac{\log n}{\overline{d}^{\gamma/2}}} \to 0 \]
for all $n$ large enough. Thus for any $\varepsilon>0$, for all $n$ large we have
\begin{align*}
\PG(\max_{i\in [n]}|m_i-t|>\varepsilon)\le &\PG(A_{n,\sqrt{\overline{d}^{1-\gamma}}})+\PG( C_{n,n\log n /\overline{d}^{\gamma}}),
\end{align*} 
which along with \eqref{eq:al} and \eqref{eq:c} gives
\begin{align*}\log \PG(\max_{i\in [n]}|m_i-t|\ge  \varepsilon)\lesssim &-\min\Big[\overline{d}^{1-\gamma},\frac{n\log n}{\overline{d}^{\gamma}}\Big]=-\overline{d}^{1-\gamma},
\end{align*}
which verifies \eqref{eq:new_claim}, and hence completes the proof of part (b)(ii).

It thus suffices to verify \eqref{eq:op}. To this effect, setting $\bJ:=\frac{1}{n}{\bf 1}{\bf 1}'$ and $\bDL$ denote a diagonal matrix with entries $\bDL_{ii}:= \text{sech}^2(\beta \xi_i)$ we have
\begin{align*}
\lVert \mathbf{D}^{(n)}\rVert_2 =\lVert \beta \bQ\bDL -\beta (1-t^2)\bJ\rVert_2 \le \beta \lVert \bQ\bDL-\bJ\bDL\rVert_2 +\beta \lVert \bJ[ \bDL -(1-t^2)\bI]\rVert_2,
\end{align*}
from which \eqref{eq:op} follows on noting that $\lVert \bQ-\bJ\rVert_2=o(1)$ by assumption, and \[\lVert \bJ [\bDL -(1-t^2)I]\rVert^2_2 =\lVert \bJ [\bDL-(1-t^2)I]^2 \bJ\rVert_2=\frac{1}{n}\sum_{i=1}^n [\bDL_{ii}-(1-t^2)]^2\to 0,\]
where the last limit uses the fact that we are working in the set $C_{n,n\log{n}/\overline{d}^{\gamma}}$. This verifies \eqref{eq:op}, and hence completes the proof of part (c).

	\emph{Part (c)(ii)}
	
%
By~\cite[Equation (4.12)]{Deb2020}, on the set $A_{n,\lambda}$ for any $\ell\ge 1$ we have:
	$$\max_{i\in [n]}\bigg|m_i-\bar{\mm}-\sum_{j=1}^n \tilde{\bQ}^\ell_{ij}(m_j-\bar{\mm})\bigg|\lesssim \lambda\ell \sqrt{\log{n}}/\sqrt{\bd},$$
	where $\tilde{\bQ}_{ij}:=\bQ_{ij}$ for $i\ne j$, and $\tilde{\bQ}_{ii}:=\frac{d_{\max}}{\overline{d}}-1$ satisfies $\btQ{\bf 1}={\bf 1}$.
	Set $\ell=D\log{n}$ for $D$ fixed but large enough so that $\max_{i\in [n]}\tilde{\bQ}^\ell_{ii}\le \frac{3}{n}$ (such a $D$ exists by~\cite[Lemma 5.2(a)]{Deb2020}). Then we have \[\Bigg|\max_{i\in [n]}\sum_{j=1}^n\tilde{\bQ}_{ij}^\ell(m_j-\bar{\mm})\Bigg|\leq \sqrt{\max_{i\in [n]}\tilde{\bQ}_{ii}}\sqrt{\sum_{j=1}^n(m_j-\bar{\mm})^2}\le \sqrt{3}\sqrt{\frac{1}{n}\sum_{i=1}^n(m_i-\bar{\mm})^2},\] and so
	\begin{align}\label{eq:basebd}
	&\;\;\;\;\PG\left(\max_{i\in [n]} \big|m_i-\bar{\mm}\big|\geq \lambda(\log{n})^{3/2}/\sqrt{\bd}+\lambda\sqrt{n^{-1}\sum_{i}\mu_i}\right)\nonumber \\ &\leq \PG(A_{n,\lambda}^c)+\PG\left(\sum_{i=1}^n (m_i-\bar{\mm})^2\gtrsim \lambda^2n(\log{n})^3/\bd+\lambda^2\sum_{i=1}^n \mu_i\right).
	\end{align}
	The desired conclusion follows from this on noting that for $\delta$ small enough, using a similar argument as~\eqref{eq:changemes}, we have 
	$$\log{\EG[\exp\left(\delta\sum_{i=1}^n (m_i-\bar{\mm})^2\right)]}\lesssim \frac{n}{\overline{d}}\log n+\sum_{i=1}^n \mu_i,$$
	where we have used~\cite[Lemma 2.2]{Deb2020}. 

	\end{proof}
%
%
\subsection{Proof of Lemma~\ref{lem:slarge}}
	    \emph{(a)} For $i\neq j$ setting $m_i^{(j)}:=\sum_{k\ne j}\bQ_{ik}X_k$ we have $\EZ [(X_j-\tanh(\beta m_j))\tanh(\beta m_i^{(j)})]=0$, and so
	    \begin{align}\label{eq:corr_1}
	    \Big|\EZ [X_iX_j]-\EZ[\tanh(\beta m_i) \tanh(\beta m_j)]\Big|=\Big|\EZ [(X_j-\tanh(\beta m_j))\tanh(\beta m_i)]\Big|\le 2\beta \bQ_{ij}.
	    \end{align}
	     Now, using part (a) of Lemma \ref{lem:mcontrol}, for any  positive integer $k$ we have 
	     \begin{align}\label{eq:moment_k}
	     \EZ\left[ \max_{i\in [n]}|m_i|^k\right]\lesssim \alpha_n^k.
	     \end{align} 
	     Using this, a Taylor's series expansion gives
	     $|\tanh(\beta m_i)-\beta m_i|\lesssim |m_i|^3$, which gives
	     \begin{align}\label{eq:bound_0}
	     \Big|\EZ\left[\tanh(\beta m_i)\tanh(\beta m_j)\right]-\beta^2\EZ\left[m_i m_j\right]\Big|\lesssim \alpha_n^4.
	     \end{align}
	     Equipped with \eqref{eq:corr_1}, \eqref{eq:moment_k} and \eqref{eq:bound_0}, we now complete the proof of part (a). To verify the first estimate, setting $\rho_{n}^{(1)}:=\max_{k\ne \ell}\EZ\left[X_k X_{\ell}\right]$ we have
	     	     \begin{align*}
	      \EZ\left[m_im_j\right]=\sum_{k,\ell=1}^n \bQ_{ik}\bQ_{j\ell}\EZ\left[X_k X_\ell\right] \le \rho_{n}^{(1)}\max_{i\in [n]}\Big(\frac{d_i}{\overline{d}}\Big)^2
	      \end{align*}
	      which along with \eqref{eq:corr_1}, \eqref{eq:moment_k} and \eqref{eq:bound_0} gives the existence of a constant $M$ such that
	        \begin{align}\label{eq:bound_part_a}
	     \rho_n^{(1)} \le \beta^2  \rho_{n}^{(1)}\max_{i\in [n]}\Big(\frac{d_i}{\overline{d}}\Big)^2+M\Big[\max_{i\ne j}\bQ_{ij}+\alpha_n^4\Big]\Rightarrow \rho_n^{(1)}\lesssim \frac{1}{\overline{d}}+\alpha_n^4\lesssim \frac{1}{\overline{d}},
	     \end{align}
	     where the last bound uses $\overline{d}\gtrsim (\log n)^2$.
	Proceeding to show the second bound, setting $\rho_n^{(2)}:=\max_{(k,\ell)\notin \mathcal{E}_n} \EZ [X_k X_\ell]$ 
	and using \eqref{eq:bound_part_a} gives the existence of a finite constant $\tilde{M}$ free of $n$ such that for all $(i,j)\notin \mathcal{E}_n^c$ we have
	    \begin{align*}
	     \EZ[m_i m_j]&=\sum_{k,\ell=1}^n \bQ_{ik}\bQ_{j\ell}\EZ [X_{k}X_{\ell}]\\
	    &\le \frac{\tilde{M}}{\overline{d}} \sum_{(k,\ell)\in \mathcal{E}_n} \bQ_{ik}\bQ_{j\ell}\textcolor{black}{\mathbf{G}_n(k,l)}+\beta^2\max_{(k,\ell)\notin \mathcal{E}_n}\E X_k X_\ell\sum_{(k,\ell)\notin \mathcal{E}_n}\bQ_{ik}\bQ_{j\ell}\\
	    & \le \tilde{M} (\bQ^3)_{ij}+\beta^2\rho_n^{(2)}\frac{d_id_j}{\overline{d}^2} =\tilde{M} (\bQ^3)_{ij}+\beta^2\rho_n^{(3)} \frac{d_id_j}{\overline{d}^2}.
	    \end{align*}
	    Since $\max_{i\in [n]}\frac{d_i}{\overline{d}}\to 1$ and $\beta<1$, 
	   using \eqref{eq:corr_1} and \eqref{eq:bound_0} along with the above display gives $\rho_n^{(2)}\lesssim  \max_{i,j}\bQ_{ij}^{3}+\alpha_n^4,$
which is the second conclusion of part (a).
	     \\
	     

	    \emph{Part (b) (i)}
	    A Taylor's series expansion of $g_i(x):=\tanh(\beta x+\mu_i)$  gives
	    \begin{align}\label{eq:taylor_11}
	    g_i(m_i)=g_i(t)+(m_i-t)g_i'(t)+\frac{g''_i(t)}{2}(m_i-t)^2+\frac{g'''_i(\xi)}{3!}(m_i-t)^3,
	    \end{align}
	    where $\xi_i$ lies between $m_i$ and $t$. Also, using part (b) of Lemma \ref{lem:mcontrol}, for any positive integer $k$ we have
	    \begin{align}\label{eq:moment_cond}
	    \EG \max_{i\in [n]}|m_i-t|^k\lesssim \alpha_n^k,
	    \end{align}
	    and so
	   \begin{align*}
	   \EG [g_i(m_i)g_j(m_j)|\bar{\bX}\geq 0]=&g_i(t) g_j(t)+ g_i'(t) g_j(t)\EG [m_i-t|\bar{\bX}\geq 0]+g_i(t)g_j'(t)\EG [m_j-t|\bar{\bX}\geq 0]\\
	    +&\frac{g_i''(t)g_j(t)}{2} \EG[(m_i-t)^2|\bar{\bX}\geq 0]+\frac{g_i(t)g_j''(t)}{2}\EG [(m_j-t)^2|\bar{\bX}\geq 0]\\
	    +&g_i(t)g_j(t)\EG[ (m_i-t)(m_j-t)|\bar{\bX}\geq 0]+O(\alpha_n^3),\\
	    \EG [g_i(m_i)|\bar{\bX}\geq 0]=&g_i(t)+g_i'(t)\EG [m_i-t|\bar{\bX}\geq 0]+\frac{g_i''(t)}{2} \EG[(m_i-t)^2|\bar{\bX}\geq 0]+O(\alpha_n^3).
	   \end{align*}
	   A direct multiplication using the last display gives
	   \begin{align}\label{eq:cond_2}
	  \Big|\CG(g_i(m_i),g_j(m_j)|\bar{\bX}\geq 0)-g_i'(t)g_j'(t)\CG (m_i-t,m_j-t|\overline{\bX}\geq 0)\Big|\lesssim \alpha_n^3.
	   \end{align} 
	 
	   We now claim that there exists $\rho>0$ such that for all $n$ large enough we have
	    \begin{align}\label{eq:cond_tanh}
	   \max_{i\ne j}\Big|\CG(X_i,X_j|\bar{\bX}\geq 0)-\CG(\tanh(\beta m_i+\mu_i),\tanh(\beta m_j+\mu_j)|\bar{\bX}\geq 0)\Big|\lesssim \bQ_{ij}+e^{-\rho n}
	    \end{align}
	
	    Note that \eqref{eq:cond_tanh}, \eqref{eq:moment_cond} and \eqref{eq:cond_2} are the analogues of \eqref{eq:corr_1}, \eqref{eq:moment_k} and \eqref{eq:bound_0}.  Given these estimates, the rest of the proof follows along similar lines as in part (a), on noting that 
	    \[\max_{1\le i,j\le n}g_i'(t)g_j'(t)=\beta^2\max_{1\le i,j\le n}\text{sech}^2(\beta t+\mu_i)\text{sech}^2(\beta t+\mu_j)=\beta^2 \text{sech}^4(\beta t+\max_{i\in [n]}\mu_i)<1-\epsilon\]
	    for some fixed $\epsilon>0$ and all large enough $n$ by assumption. It only remains to verify \eqref{eq:cond_tanh}. To this effect, we first claim that there exists a (different) constant $\rho>0$ free of $n$, such that for any function $f:\{-1,1\}^n\mapsto[-1,1]$ we have
	    \begin{align}\label{eq:margin}
		\max_{i\in [n]}\bigg|\EG \bigg[f(\bX)|\bar{\bX}\geq 0\bigg]-\EG \bigg[f(\bX)|\bar{\bX}_{i}>0\bigg]\bigg|\le 5e^{-\rho n},
		\end{align}
		where $\bar{\bX}_i:=\frac{1}{n}\sum_{j\ne i}X_j$.
	    Given \eqref{eq:margin}, noting that $\PG (\bar{\bX}\geq 0)\ge 1/3$ for all large enough $n$, we have
	    \begin{align*}
	    \ & \EG (X_iX_j|\bar{\bX}\geq 0)\\
	    =&\frac{\EG \Big[X_iX_j\mathbf{1}\{\bar{\bX}\geq 0\}\Big]}{\PG(\bar{\bX}\geq 0)}\\
	    =&\frac{\EG [X_iX_j\mathbf{1}\{\bar{\bX}_{j}>0\}]}{\PG(\bar{\bX}\geq 0)}+O(e^{-\rho n})\text{  [By \eqref{eq:margin}]}\\
	    =&\frac{\EG[ \tanh(\beta m_i+\mu_i) X_j \mathbf{1}\{\bar{\bX}_j>0\}]}{\PG(\bar{\bX}\geq 0)}+O(e^{-\rho n})\\
	    =&\frac{\EG[ \tanh(\beta m_i^{(j)}+\mu_i) X_j \mathbf{1}\{\bar{\bX}_i>0\}]}{\PG (\bar{\bX}\geq 0)}+O(e^{-\rho n})+O(\bQ_{ij})\text{  [By \eqref{eq:margin}]}\\
	    =&\frac{\EG [\tanh(\beta m_i^{(j)}+\mu_i) \tanh(\beta m_j+\mu_j)\mathbf{1}\{\bar{\bX}_i>0]}{\PG (\bar{\bX}\geq 0)}+O(e^{-\rho n})+O(\bQ_{ij})\\
	    =&\frac{\EG [\tanh(\beta m_i+\mu_i) \tanh(\beta m_j+\mu_j)\mathbf{1}\{\bar{\bX}\geq 0]}{\PG (\bar{\bX}\geq 0)}+O(e^{\rho n})+O(\bQ_{ij})\text{  [By \eqref{eq:margin}]}\\
	    =&\EG [\tanh(\beta m_i+\mu_i) \tanh(\beta m_j+\mu_j)|\bar{\bX}\geq 0]+O(e^{-\rho n})+O(\bQ_{ij}).
	    \end{align*}
	    A similar calculation gives 
	    \begin{align}\label{eq:var_cond}
	    \EG [X_i|\bar{\bX}\geq 0]=\EG [\tanh(\beta m_i+\mu_i)|\bar{\bX}\geq 0]+O(e^{-\rho n})+O(\bQ_{ij}),
	    \end{align} which along with the above display gives
	    \eqref{eq:cond_tanh}, as desired. To complete the proof, we need to verify \eqref{eq:margin}. To this effect, use \cite[Equation (2.8)]{Deb2020} to note that\[\PG(\bar{\bX}\geq 0,\bar{\bX}_i\le 0)+\PG(\bar{\bX}\le 0,\bar{\bX}_i>0)\le e^{-\rho n}\Rightarrow |\PG(\bar{\bX}\geq 0)-\PG(\bar{\bX}_i>0)|\le e^{-\rho n},\]
and so
\begin{align*}
	&\Big|\EG[f(\bX)|\bar{\bX}\geq 0]-\EG[f(\bX)|\bar{\bX}_i>0]\Big|\\
	\le &\Big|\frac{1}{\PG(\bar{\bX}\geq 0)}-\frac{1}{\PG(\bar{\bX}_i>0)}\Big|+\PG(\bar{\bX}\geq 0, \bar{\bX}_i<0)+\PG(\bar{\bX}<0, \bar{\bX}_i>0)
	\le 5e^{-\rho n},
	\end{align*}
	    which verifies \eqref{eq:margin}.
	    
	   \emph{Part (b)(ii)}  
	It suffices to show that
	$\max_{i\in [n]}\Big|\EG(X_i|\bar{\bX}\geq 0)-\tanh(\beta t+\mu_i)\Big|\lesssim \alpha_n.$
	But this follows on using \eqref{eq:taylor_11} and \eqref{eq:moment_cond} to note that
	\[\Big| \EG (\tanh(\beta m_i+\mu_i)|\bar{\bX}\geq 0)-\tanh(\beta t +\mu_i)\Big|\lesssim \EG |m_i-t|\lesssim  \alpha_n.\]

 \emph{Part (c).} To begin, use \cite[Equation (4.25)]{Deb2020} and \cite[Lemma 2.4 (c)]{Deb2020}, coupled with the assumption $\bd\gg \sqrt{n}(\log{n})^5$ to note that
 \[ \EZ \bar{\mm}^2\lesssim \frac{1}{n}+\E \bar{\bX}^2\lesssim \frac{1}{n}+\frac{1}{\sqrt{n}}\lesssim \frac{1}{\sqrt{n}}.\]
Also, using part (c)(ii) of Lemma \ref{lem:mcontrol} with $\bmu={\bf 0}$ gives
 \begin{align*}
  &|\EZ \tanh(m_i)\tanh(m_j)|\\ =&\left|\EZ \bigg[\left(\tanh( m_i)-\tanh(\bar{\mm})+\tanh(\beta \bar{\mm})\right)\left(\tanh( m_j)-\tanh(\bar{\mm})+\tanh(\beta \bar{\mm})\right)\bigg]\right|\\
 \le& \EZ|(m_i-\bar{\mm})(m_j-\bar{\mm})|+\EZ|(m_i-\bar{\mm}) \bar{\mm}|+\EZ|(m_j-\bar{\mm})\bar{\mm}|+\EZ\bar{\mm}^2\\
 \le &\EZ \max_{i\in [n]}|m_i-\bar{\mm}|^2+2\sqrt{\EZ \max_{i\in [n]}|m_i-\bar{\mm}|^2}\sqrt{\EZ \bar{\mm}^2}+\EZ\bar{\mm}^2\\
 \lesssim & \alpha_n^2+2\frac{\alpha_n}{n^{1/4}}+\frac{1}{\sqrt{n}}\lesssim \frac{1}{\sqrt{n}}.
 \end{align*}

%
%
%
%
%
%

\subsection{Proof of Lemma~\ref{lem:smallinsig}}
	\emph{Part (a)}. To begin, note that for any $i\in [n]$ we have
	\[\PZ (X_i=x_i|X_j=x_j,j\ne i)=\frac{e^{\beta m_i}}{e^{\beta m_i}+e^{-\beta m_i}}\ge \frac{1}{1+e^{-2C_u'\beta}}=:p\]
	where $\max_{i\in [n]} |m_i(\mathbf{X})|\leq \max_{1\leq i\leq n}\sum_{j=1}^n \bQ_{ij}\leq C_u'$. The above display on taking expectation gives $\PZ (X_i=x_i|X_j=x_j, j\in A)\ge p$ for any $A\subset [n]$, and so 
	$\PZ (\bX_{S}={\bf a})\ge p^{s}$. On the other hand, setting \[m_i(S):=\sum_{j\in S^c}Q_{ij}x_j,\quad \Omega(S):=\left\{{\bf x}\in \{-1,1\}^{n-s}:\max_{i\in [n]}|m_i(S)|\le \lambda \frac{\log n}{\sqrt{\overline{d}}}\right\},\] we have \[|m_i-m_i(S)|\le \frac{s}{\overline{d}}\lesssim \frac{\log n}{\overline{d}}\Rightarrow\PZ (\bX_{S^c}\notin \Omega(S))\le n^{-\rho\lambda}\] for some $\rho>0$,
	where we use part (a) of Lemma \ref{lem:mcontrol}. Consequently, for $\lambda$ sufficiently large, for any ${\bf a}\in \{-1,1\}^s$ we have
	\begin{align*}
	\ & \PZ (\bX_S={\bf a})\ge p^{s}\ge p^{c\log n}\gg n^{-\rho\lambda},\text{ and so }\\
	&\lim_{n\rightarrow\infty}\sup_{S:|S|=s,\ {\bf a}\in \{-1,1\}^s}\Big|\frac{\PZ (\bX_S={\bf a})}{\PZ (\bX_S={\bf a},\bX_{S^c}\in \Omega(S))}-1\Big|=0.
	\end{align*}
	It thus suffices to estimate $\PZ (\bX_S={\bf a},\bX_{S^c}\in \Omega(S))$. To this effect, we have
	\begin{align}
	\notag&\PZ (\bX_S={\bf a},\bX_{S^c}\in \Omega(S))\\
	\notag&\frac{1}{\ZZ}\sum_{{\bf x}\in \Omega(S)} \exp\Big(\frac{\beta}{2}\sum_{i,j\in S^c}\bQ_{ij}x_ix_j+\beta\sum_{i\in S, j\in S^c}\bQ_{ij}a_ix_j+\frac{\beta}{2}\sum_{i,j\in S}\bQ_{ij}a_ia_j\Big)\\
	\notag= &\frac{1}{\ZZ}\sum_{{\bf x}\in \Omega(S)} \exp\Big(\frac{\beta}{2}\sum_{i,j\in S^c}\bQ_{ij}x_ix_j+\beta\sum_{ i\in S}a_im_i(S)+\frac{\beta}{2}\sum_{i,j\in S}\bQ_{ij}a_ia_j\Big)\\
	\label{eq:upp_infty}\le &\frac{\exp(\lambda\beta s \frac{\log n}{\sqrt{\bd}}+\frac{\beta s^2}{2\bd}\Big)}{\ZZ}\sum_{{\bf x}\in\Omega(S)}\exp\Big(\frac{\beta}{2}\sum_{i,j\in S^c}\bQ_{ij}x_ix_j\Big).
	\end{align}
	A similar calculation gives
	\begin{align}\label{eq:low_infty}
	\PZ (\bX_S={\bf a},\bX_{S^c}\in \Omega(S))\ge &\frac{\exp(-\lambda\beta s \frac{\log{n}}{\sqrt{\bd}}-\frac{\beta s^2}{2\bd}\Big)}{\ZZ}\sum_{{\bf x}\in\Omega(S)}\exp\Big(\frac{\beta}{2}\sum_{i,j\in S^c}\bQ_{ij}x_ix_j\Big).
	\end{align}
	Note that both the bounds in \eqref{eq:upp_infty} and \eqref{eq:low_infty} are free of ${\bf a}$ and depend on $S$ only through its cardinality, i.e., $s$. The ratio of these two bounds converge to $1$ using the fact that $\bd \gg (\log n)^4$.  The conclusion in~\eqref{eq:smallinsig} then follows.
	
	\emph{Part (b)} The proof of part (a) goes through verbatim after replacing the term $\frac{\log n}{\sqrt{\overline{d}}}$ in the definition of $\Omega(S)$ by $\frac{(\log n)^{2/3}}{\overline{d}^{1/6}}$.

	\emph{Part (c)}. Again the proof is  similar, except that we now use the bound $\sum_{i\in S}|a_im_i(S)-a_it|\le \lambda s \frac{\log n}{\sqrt{\overline{d}}}$ for ${\bf x}\in \Omega(S)$, where the revised $\Omega(S)$ is defined as:
	$$\Omega(S):=\left\{{\bf x}\in \{-1,1\}^{n-s}:\max_{i\in [n]}|m_i(S)-t|\le \lambda \frac{\log n}{\sqrt{\bar d}},\ \frac{1}{n}\sum_{i\in S^c} X_i>0\right\}.$$

\subsection{Proof of Lemma~\ref{lem:altbeh}}

Since the probability distribution $\PG$ is monotonic in $\bmu$ (coordinate-wise), without loss of generality we can replace $\mu$ by $\tilde{\mu}$, where $\tilde{\mu}_i=\min(A,\frac{\sqrt{\bar d}}{s\sqrt{\log n}})$ for $i\in S$ and $\tilde{\mu}_i=0$ otherwise. Then we have
\[ \sum_{i=1}^n\tilde{\mu}_i=\min\Big(sA, \sqrt{\frac{\overline{d}}{\log n}}\Big),\Rightarrow n^{1/4}\ll \sum_{i=1}^n\tilde{\mu}_i\le \sqrt{\frac{\overline{d}}{\log n}}. \]
Therefore, more generally, we will show the existence of $\eta>0$ such that 
\begin{align}\label{eq:maintailbd}
\PG(n\overline{\bX}^3> \eta \sum_{i=1}^n\mu_i)\to 1\text{ whenever }
n^{1/4}\ll \sum_{i=1}^n\mu_i\lesssim \sqrt{\frac{\overline{d}}{\log n}},\quad \max_{i\in [n]}\mu_i\to 0.
\end{align}
This choice gives
\[\sum_{j=1}^nQ_{ij}\mu_j\le \frac{1}{\overline{d}}\sum_{j=1}^n\mu_j\le \frac{\sqrt{\overline{d}}}{\overline{d}\sqrt{\log n}}=\frac{1}{\sqrt{\overline{d}\log n}}\ll \sqrt{\frac{\log n}{\overline{d}}},\]
and so Lemma \ref{lem:mcontrol} part (c)(ii) applies.

We begin by claiming the following, whose proofs we defer.
	\begin{eqnarray}\label{eq:mainclaim1}
	\EG(n^{1/4}\overline{\bX})^6&\lesssim & n^{-1/2}\left(\sum_{i=1}^n \mu_i\right)^2,\\
	\label{eq:mainclaim2}\EG\left[\sum_{i=1}^n (d_i/\bd-1)X_i\right]^2&\lesssim& n^{1/3}\left(\sum_{i=1}^n \mu_i\right)^{2/3} ,\\
\label{eq:mainclaim3}\EG\left[\sum_{i=1}^n (X_i-\tanh(m_i+\mu_i))\right]^2&\lesssim &n^{1/3}\left(\sum_{i=1}^n \mu_i\right)^{2/3}.
	\end{eqnarray}
	An application of the triangle inequality gives:
	\begin{align}\label{eq:lowerbd}
	\sum_{i=1}^n (X_i-\tanh(m_i))&\geq \sum_{i=1}^n (\tanh(m_i+\mu_i)-\tanh(m_i))-\Bigg|\sum_{i=1}^n (X_i-\tanh(m_i+\mu_i))\Bigg|\nonumber \\ &\geq \rho\sum_{i=1}^n \mu_i-\Bigg|\sum_{i=1}^n (X_i-\tanh(m_i+\mu_i))\Bigg|
	\end{align}
	for some positive constant $\rho$ free of $n$. 
	Also, since $\sum_{i=1}^n \mu_i\gg n^{1/4}$, using \eqref{eq:mainclaim3} gives
	\begin{align}\label{eq:lowerbd2}
	\Bigg|\sum_{i=1}^n(X_i-\tanh(m_i+\mu_i))\Bigg|=O_p\left(n^{1/3}\left(\sum_{i=1}^n\mu_i\right)^{2/3}\right)=o_p\left(\sum_{i=1}^n\mu_i\right).
	\end{align}
	Finally, a Taylor's series expansion of $\tanh(m_i)$ at $\bar{\mm}$ gives
	\begin{align}\label{eq:taylor_2020}
	&\Big|\sum_{i=1}^n\tanh(m_i)-n\tanh(\bar{\mm})\Big|\lesssim |\bar{\mm}|\sum_{i=1}^n(m_i-\bar{\mm})^2+\sum_{i=1}^n|m_i-\bar{\mm}|^3,
	\end{align}
	and so
\begin{align}\label{eq:final}
\notag&\Big|\sum_{i=1}^n(X_i-\tanh(m_i))\Big|\\
\notag\lesssim& n|\bar{\bX}|^3+n|\bar{\mm}-\bar{\bX}|+|\bar{\bX}|\sum_{i=1}^n(m_i-\bar{\mm})^2+\sum_{i=1}^n|m_i-\bar{\mm}|^3\\
=& n\overline{\bX}^3+O_p\left(n^{1/6}\bigg(\sum_{i=1}^n \mu_i\bigg)^{1/3}\right)+O_p\left(n^{-1/3}\bigg(\sum_{i=1}^n \mu_i\bigg)^{1/3} \left(\frac{n(\log{n})^3}{\bd^2}+n\overline{\bmu}\right)\right)\nonumber\\
+& O_p\left(\frac{n(\log{n})^{9/2}}{\bd^{3/2}}+n(\overline{\bmu})^{3/2}\right)\nonumber \\ =&n\overline{\bX}^3+o_p\left(\sum_{i=1}^n \mu_i\right),
	\end{align}
	
	where the last line uses~\eqref{eq:mainclaim1},~\eqref{eq:mainclaim2}~and Lemma \ref{lem:mcontrol} part (c)(ii). Combining \eqref{eq:final} along with~\eqref{eq:lowerbd} and \eqref{eq:lowerbd2} completes the proof of~\eqref{eq:maintailbd}.

	
	We now verify the three claims \eqref{eq:mainclaim1}, \eqref{eq:mainclaim2}, \eqref{eq:mainclaim3}. To this effect,  note that
		\begin{align}\label{eq:mainclaim21}
	&\EG(n^{1/4}\bar{\bX})^6\lesssim \frac{1}{\sqrt{n}}\left\{\left(\sum_{i=1}^n \mu_i\right)^2+n^{-1/2}\EG\bigg[\sum_{i=1}^n (d_i/\bd-1)X_i\bigg]^2\right\},\\
	&\label{eq:mainclaim22}	\EG\left[\sum_{i=1}^n (d_i/\bd-1)X_i\right]^2\lesssim \sqrt{n}[1+\EG(n^{1/4}\bar{\bX})^2].
	\end{align}
	together imply \eqref{eq:mainclaim1} and \eqref{eq:mainclaim2}. It thus suffices to verify \eqref{eq:mainclaim21},  \eqref{eq:mainclaim22} and \eqref{eq:mainclaim3}.
	\begin{itemize}
	\item{Proof of \eqref{eq:mainclaim21}}

	The proof follows closely the proof of~\cite[Equation 4.13]{Deb2020}.
	
	Set $T_n=n^{-3/4}\sum_{i=1}^n X_i$, and form an exchangeable pair $(\mathbf{X},\mathbf{X}')$ as follows: Let $I$ denote a randomly sampled index from $\{1,2,\ldots ,n\}$. Given $I=i$, replace $X_i$ with an independent $\pm 1$ valued random variable $X_i'$ with mean $\tanh(\beta m_i+\mu_i)=\EG[X_i|(X_j,j\neq i)]$, and let $\mathbf{X}':=(X_1,\cdots,X_{i-1},X_i',X_{i+1},\cdots,X_n)$. Then we have
	$$\EG[T_n-T_n'|\mathbf{X}]=\frac{1}{n^{7/4}}\sum_{i=1}^n (X_i-\tanh(m_i+\mu_i))=\frac{1}{n^{7/4}}\sum_{i=1}^n(X_i-\tanh(m_i))+\frac{1}{n^{7/4}}\sum_{i=1}^n \xi_i\mu_i,$$
	where $\{\xi_i\}_{1\le i\le n}$ are bounded random variables. This, along with the last display gives
	\begin{align*}
	&\;\;\;\big|\EG[T_n-T_n'|\mathbf{X}]-n^{-3/2}T_n^3/3\big|\nonumber \\ &\le \frac{2}{15} n^{-2}|T_n|^5+M\bigg\{n^{-3/4}|\bar{\bX}-\bar{\mm}|+n^{-2}|T_n|\sum_{i=1}^n (m_i-\bar{\mm})^2+n^{-7/4}\bigg|\sum_{i=1}^n (m_i-\bar{\mm})^3\bigg|\bigg\}
	\end{align*}
	for some fixed constant $M>0$.
	On multiplying both sides of the above inequality by $|T_n|^3$ and taking expectation gives
	\begin{align*}
	 \ & \EG[T_n^6]
	\\&\leq  (2/5)n^{-1/2}\EG|T_n|^8\\
	&+3M\left\{\begin{array}{c} n^{3/4}\EG\left[|T_n|^3|\overline{\bX}-{\overline{\mm}}|\right]+n^{-1/2}\EG\left[|T_n|^4\sum_{i=1}^n (m_i-\overline{{\mm}})^2\right] \\+n^{-1/4}\EG\left[|T_n|^3\big|\sum_{i=1}^n (m_i-\overline{{\mm}})^3\big|\right]+n^{-1/4}\EG|T_n|^3\sum_{i=1}^n \mu_i\end{array}\right\}\\&+3n^{3/2}\big|\EG(T_n-T_n')T_n^3\big|.
	\end{align*}
	This is the analogue of~\cite[Equation 4.15]{Deb2020} in the case when $\bmu$ is not necessarily ${\bf 0}$. Also, using part (c)(ii) of Lemma \ref{lem:mcontrol}  we have
	\begin{align}\label{eq:qformuni}
	 \EG\Big[\sum_{i=1}^n(m_i-\bar{\mm})^2\Big]^p \lesssim  (n\alpha_n^2+n\overline{\bmu})^p,\quad \EG\max_{1\le i\le N}|m_i-\bar{\mm}|^p \lesssim \left(\alpha_n+\sqrt{\overline{\bmu}}\right)^p,
	\end{align}
	where we use the fact that $\bar{\bmu}\lesssim \frac{\overline{d}}{n}\le \alpha_n$, as $\overline{d}\gg \sqrt{n}$. This is the analogue of \cite[Equation 4.17]{Deb2020}. Hereon, proceeding similarly as in the derivation of~\cite[Equation 4.13]{Deb2020} gives \eqref{eq:mainclaim21}.
	

\item{Proof of \eqref{eq:mainclaim22}}

	This proof is similar to the derivation of~\cite[Equation 4.14]{Deb2020}. 
	
	With $\tilde{\bQ}$ as defined in the proof of Lemma \ref{lem:mcontrol} part (c)(ii), 
	set $\mathbf{c}^{\top}:=(d_1/\bd-1,\ldots ,d_n/\bd-1)$,  $(\mathbf{c}^{(\ell)})^{\top}:=\mathbf{c}^{\top}(\tilde{Q})^\ell$ and $x_\ell:=\EG[\sum_{i=1}^n c^{(\ell)}_iX_i]^2$. Note that, we can write $x_\ell=T_{1\ell}+T_{2\ell}+T_{3\ell}$ where,
	\begin{align*}
	T_{1\ell}:=\EG\Bigg[& \sum_{i=1}^n c^{(l)}_i(X_i -\tanh(m_i+\mu_i))\Bigg]^2, \qquad T_{2\ell}:=\EG\left[\sum_{i=1}^n c^{(l)}_i\tanh(m_i+\mu_i)\right]^2\\
	&T_{3\ell}:=2\EG\left[\sum_{i\neq j} c^{(\ell)}_i c^{(\ell)}_j(X_i-\tanh(m_i+\mu_i))\tanh(m_i+\mu_i)\right].
	\end{align*}
	Using Lemma \ref{lemma:chatterjee} it follows that 
	\begin{align}\label{eq:t1}
	T_{1\ell}\lesssim \lVert \mathbf{c}^{(\ell)}\rVert_2^2\le \lVert \mathbf{c}\rVert_2^2.
	\end{align}
	For controlling $T_{3\ell}$ setting $m_i^{(j)}:=\sum_{k\ne j}\bQ_{ik}X_k$ as before we have
	\begin{align}
	\notag|T_{3\ell}|=&2\left|\sum_{i\neq j} c^{(\ell)}_i c^{(\ell)}_j\EG(X_i-\tanh(m_i+\mu_i))(\tanh(m_i+\mu_i)-\tanh(m_i^{(j)}+\mu_i))\right|\\
	\lesssim &\sum_{i\ne j}|c^{(\ell)}_i| |c^{(\ell)}_j|\bQ_{ij}\lesssim \lVert \mathbf{c}^{(\ell)}\rVert_2^2\le \lVert \mathbf{c}\rVert_2^2.\label{eq:t3}
	\end{align}
	For bounding $T_{2\ell}$, note that
	\begin{align*}
	 \ & n|\bar{\bX}-\bar{\mm}|\\
	 =&\Bigg|\sum_{i=1}^nc_iX_i\Bigg|\\
	\le& \Bigg|\sum_{i=1}^nc_i(X_i-\tanh(m_i+\mu_i))\Bigg|+\Bigg|\sum_{i=1}^nc_i(\tanh(m_i+\mu_i)-\tanh(\bar{\mm}+\mu_i))\Bigg|+\Bigg|\sum_{i=1}^nc_i\tanh(\mu_i)\Bigg|\\
	\lesssim &\Bigg|\sum_{i=1}^nc_i(X_i-\tanh(m_i+\mu_i))\Bigg|+\lVert {\mathbf c}\rVert_2\sqrt{\sum_{i=1}^n(m_i-\bar{\mm})^2}+n\alpha_n\bar{\bmu},
	\end{align*}
	where the last step uses the bound $\max_{i\in [n]}|c_i|=\max_{i\in [n]}\Big|\frac{d_i}{\overline{d}}-1\Big|\lesssim \alpha_n$. Consequently, for any positive integer $p$ using \eqref{eq:qformuni} we have:
	\begin{align}\label{eq:basestimate1}
	\EG(\overline{\bX}-\overline{{\mm}})^{2p}\le&
	\notag n^{-2p}\left(\lVert \mc\rVert_2^{2p}(1+n^p(\alpha_n^2+\overline{\bmu})^{p})+(n\alpha_n\bar{\bmu})^{2p}\right)\\
	\le& n^{-2p}\left[n^p\alpha_n^{2p}(1+n^p\alpha_n^{2p}+n^p\overline{\bmu}^p)+\alpha_n^{2p}\overline{d}^p\right]\lesssim \frac{(\log n)^{2p}}{\overline{d}^{2p}}\lesssim \frac{1}{n^p},
	\end{align}
	where the last line uses the bound
	\[\lVert \mathbf{c}\rVert_2^2=\sum_{i=1}^n\Big(\frac{d_i}{\overline{d}}-1\Big)^2\lesssim n\alpha_n^2,\quad \sum_{i=1}^n\mu_i\le \sqrt{\overline{d}},\quad \overline{d}\ge \sqrt{n}\log n.\]
	In the subsequent proof, unless otherwise stated,~\eqref{eq:basestimate1} will always be invoked with $p=1$. Combining~\eqref{eq:basestimate1}~and~\eqref{eq:mainclaim1}, we get:
	$$\nu_n\lesssim 1+n^{-1/2}\left(\sum_{i=1}^n\mu_i\right)^2+\frac{n^{3/2}(\log n)^2}{\overline{d}^2}\lesssim \sqrt{n}.$$
	which again on invoking \eqref{eq:basestimate1} (with $p=3$) gives
	\begin{align}\label{eq:basestimate2}
	\EG\overline{\bX}^6=\frac{\nu_n}{n^{3/2}}\lesssim n^{-1},\; \bmme\overline{\bmu}^6\lesssim n^{-1},\; \EG\left[\sum_{i=1}^n m_i^6\right]\lesssim 1.
	\end{align}
	Now, a Taylor's series expansion gives $\tanh(m_i+\mu_i)=\tanh(m_i)+\mu_i\xi_i$ for bounded random variables $\xi_i$, and so
	\begin{align}\label{eq:basecal}
	T_{2\ell}&=\EG\left[\sum_{i=1}^n c^{(\ell)}_i\tanh(m_i)\right]^2+2\left\{\EG\left[\sum_{i=1}^n c^{(\ell)}_i\tanh(m_i)\right]\right\}\left(\sum_{i=1}^n c^{(\ell)}_i\xi_i\mu_i\right)\nonumber\\&+\left(\sum_{i=1}^n c^{(\ell)}_i\xi_i\mu_i\right)^2.
	\end{align}
	Setting $\theta_n:=1+\EG(n^{1/4}\overline{\bX})^2$ and invoking \eqref{eq:basestimate1} and \eqref{eq:qformuni} the terms in the RHS of~\eqref{eq:basecal} can be estimated as
	\begin{align*}
	\EG\left[\sum_{i=1}^n c_i^{(\ell)}\tanh(m_i)\right]^2\le &\lVert \mathbf{c}^{(\ell)}\rVert_2^2\cdot \Big[\E_\mu\sum_{i=1}^n(m_i-\bar{\mm})^2+n\EG\bar{\mm}^2\Big]
	\lesssim \lVert \mathbf{c}\rVert_2^2\sqrt{n}\theta_n,\\
	\bigg|\sum_{i=1}^n c^{(\ell)}_i\xi_i\mu_i\bigg|\leq &\max_{i\in [n]} \sum_{i=1}^n |c_i^{(\ell)}|\cdot \sum_{i=1}^n \mu_i\le \max_{i\in [n]}|c_i| \sum_{i=1}^n\mu_i\lesssim 1. 
	\end{align*}
%
	Combining the above estimate with~\eqref{eq:basestimate1}~and~\eqref{eq:basestimate2}~and repeating the derivation of~\cite[(4.32)]{Deb2020}, we get the existence of $M<\infty$ such that for all $\ell\ge 1$ we have
	$$x_{\ell}\leq x_{\ell+1}+2M\sqrt{x_{\ell+1}}\beta_n+M^2\beta_n^2,\quad \beta_n:=1+\lVert \mc\rVert_2\sqrt{\theta_n}$$
	The above relation is similar to~\cite[(4.32)]{Deb2020}. 
	Proceeding in a similar manner, setting $L=D(\log n)^2$ with $D$ large enough, an inductive argument gives $x_{\ell}\le (L-\ell+1)^2 M^2\beta_n^2$, giving
	\[x_0=\EG\Big[\sum_{i=1}^n\Big(\frac{d_i}{\overline{d}}-1\Big)X_i\Big]^2\le(L+1)^2 M^2 \beta_n^2\lesssim (\log n)^4(1+\lVert \mc\rVert_2^2\theta_n),\]
	which verifies \eqref{eq:mainclaim22}.
	\item{Proof of \eqref{eq:mainclaim3}}
	A direct expansion gives
	\begin{align*}
	&\;\;\;\EG\left[\sum_{i=1}^n (X_i-\tanh(m_i+\mu_i))\right]^2\\ &=\EG\left[\sum_{i=1}^n \sech^2(m_i+\mu_i)\right]\\
	&+\EG\left[\sum_{i\neq j} (X_i-\tanh(m_i+\mu_i))(\tanh(m_j^i+\mu_j)-\tanh(m_j+\mu_j))\right]\\ &{=} \EG\left[\sum_{i=1}^n \sech^2(m_i+\mu_i)\right]+\EG\left[\sum_{i\neq j} (1-X_i\tanh(m_i+\mu_i))(-\bQ_{ji}\sech^2(m_j^i+\mu_j))\right]\\&+O\left(\sum_{i,j=1}^n\bQ_{ij}^2\right)\\ &{=}\EG\left[\sum_{i=1}^n \sech^2(m_i+\mu_i)\left(1-\sum_{j=1}^n \bQ_{ji}\sech^2(m_j+\mu_j)\right)\right]+O\left(\sum_{i,j=1}^n\bQ_{ij}^2\right),
	\end{align*}
	The first term of the above display,  splits into two terms as follows:
	\begin{align*}
	&\;\;\;\EG\left[\sum_{i=1}^n \sech^2(m_i+\mu_i)\left(1-\frac{d_i}{\bd}\right)\right]+\EG\left[\sum_{i,j} \bQ_{ij}\sech^2(m_i+\mu_i)\tanh^2(m_j+\mu_j)\right]\\ &\overset{(a)}{\lesssim} \EG\left[\Bigg|\sum_{i=1}^n (\sech^2(\overline{\mm})+\xi_{i1}\mu_i+\xi_{i2}(m_i-\overline{\mm}))\left(1-\frac{d_i}{\bd}\right)\Bigg|\right]+\EG\left[\sum_{i,j} \bQ_{ij}\tanh^2(m_j+\mu_j)\right]\\ &\lesssim\max_i \Bigg|\frac{d_i}{\bd}-1\Bigg|\sum_{i=1}^n \mu_i+\sum_{i=1}^n \EG \Bigg|\left(\frac{d_i}{\bd}-1\right)(m_i-\overline{\mm})\Bigg|+\EG\left[\sum_{i,j} \bQ_{ij}(m_j^2+\mu_j^2)\right]\\ 
	&\lesssim 1+\sqrt{\sum_{i=1}^n\left(\frac{d_i}{\overline{d}}-1\right)^2 }\sqrt{\EG \sum_{i=1}^n(m_i-\bar{\mm})^2}+\EG\left[\sum_{i=1}^n (m_i-\overline{\mm})^2\right]+n\EG\overline{\mm}^2+\sum_{j=1}^n \mu_j^2\\
	&\overset{(b)}{\lesssim} 1+\sqrt{\frac{n\log n}{\overline{d}}}\cdot \sqrt{\frac{n (\log n)^3}{\overline{d}}+\sum_{i=1}^n\mu_i}+\frac{n (\log n)^3}{\overline{d}}+\sum_{i=1}^n\mu_i+n\EG\left(\overline{\bX}-\overline{\mm}\right)^2+n\EG[\overline{\bX}]^2\\ &+\left(\max_{j}\mu_j\right)\sum_{j=1}^n \mu_j\lesssim \sqrt{n}+\EG[\overline{\bX}]^2\lesssim n^{1/3}\left(\sum_{i=1}^n \mu_i\right)^{2/3}. 
	\end{align*}
Here (a) follows from standard Taylor expansions. Note that $\xi_{i1}$ and $\xi_{i2}$ are uniformly bounded random variables. The bounds in (b), (c) and (d) are consequences of Lemma~\ref{lem:mcontrol} part (c)(ii),~\eqref{eq:basestimate1} and~\eqref{eq:mainclaim1} respectively.
%
%
	\end{itemize}

%
%
%
%
%
%
%
%
%
%
%
%
%
%
%
%
%
%

\bibliographystyle{plainnat}
\bibliography{biblio_structured}
\end{document}